\documentclass{amsart}
\usepackage{amsmath,amscd}
\usepackage{amssymb}
\usepackage{bbm}
\usepackage{enumerate}
\usepackage{amssymb}
\usepackage{array}
\usepackage{graphicx}
\usepackage{amsthm}
\usepackage{units} 
\usepackage{dsfont}
\usepackage[arrow, matrix, curve]{xy}
\usepackage[left,modulo]{lineno}
\usepackage{tikz-cd}
\usepackage{mathrsfs}
\usepackage{amscd}

\newtheorem{theorem}{Theorem}[section] 
\newtheorem{lemma}[theorem]{Lemma}
\newtheorem{corollary}[theorem]{Corollary}
\newtheorem{proposition}[theorem]{Proposition}

\theoremstyle{definition}

\theoremstyle{remark}

\numberwithin{equation}{section}

\newcommand\numberthis{\addtocounter{equation}{1}\tag{\theequation}}

\begin{document}

\title{Wave Front Sets of Reductive Lie Group Representations II}

\author{Benjamin Harris}
\address{Department of Mathematics, Bard College at Simon's Rock, Great Barrington, Massachusetts 01230}
\email{Benjamin.Harris@simons-rock.edu}
\thanks{The author was an NSF VIGRE postdoc at Louisiana State University while this research was conducted.}



\subjclass[2010]{22E46, 22E45, 43A85}

\date{March 25, 2017}


\keywords{Wave Front Set, Singular Spectrum, Analytic Wave Front Set, Reductive Lie Group, Real Reductive Algebraic Group, Induced Representation, Tempered Representation, Branching Problem, Discrete Series, Reductive Homogeneous Space}

\begin{abstract}
In this paper it is shown that the wave front set of a direct integral of singular, irreducible representations of a real, reductive algebraic group is contained in the singular set. Combining this result with the results of the first paper in this series, the author obtains asymptotic results on the occurrence of tempered representations in induction and restriction problems for real, reductive algebraic groups.
\end{abstract}

\maketitle

\section{Introduction}
This is the second in a series of papers on wave front sets of reductive Lie group representations. The first was a joint paper with Hongyu He and Gestur \'{O}lafsson \cite{HHO}. 

Let $G$ be a Lie group, and let $(\pi,V)$ be a unitary representation of $G$. We define the \emph{wave front set} of $\pi$ to be
$$\operatorname{WF}(\pi)=\overline{\bigcup_{u,v\in V}\operatorname{WF}_e(\pi(g)u,v)}$$
and the \emph{singular spectrum} of $\pi$ to be
$$\operatorname{SS}(\pi)=\overline{\bigcup_{u,v\in V}\operatorname{SS}_e(\pi(g)u,v)}.$$
If $f$ is a continuous function on a Lie group $G$, $\operatorname{WF}_e(f)$ (resp. $\operatorname{SS}_e(f)$) denotes the piece of the wave front set (resp. singular spectrum) of $f$ in the fiber over the identity in $iT^*G$. If $\mathfrak{g}$ denotes the Lie algebra of $G$, note that both the wave front set and singular spectrum of $\pi$ are closed, invariant cones in $i\mathfrak{g}^*$. These ideas were first introduced by Kashiwara-Vergne \cite{KV} and Howe \cite{How}.

Let $G$ be a real, reductive algebraic group, and let $\widehat{G}$ denote the space of irreducible, unitary representations of $G$ equipped with the Fell topology. Let $\widehat{G}_{\text{temp}}^{\text{\ }\prime}\subset \widehat{G}$ be the open subset of irreducible, tempered representations of $G$ with regular infinitesimal character. Define
$$\widehat{G}_s=\widehat{G}-\widehat{G}_{\text{temp}}^{\text{\ }\prime}$$ and call $\widehat{G}_s$ the set of \emph{singular} irreducible, unitary representations of $G$. Similarly, let $i(\mathfrak{g}^*)'$ denote the set of regular, semisimple elements in $i\mathfrak{g}^*$. We analogously call $$i\mathfrak{g}_s^*=i\mathfrak{g}^*-i(\mathfrak{g}^*)'$$
the set of \emph{singular} elements in $i\mathfrak{g}^*$.

If $\pi$ is a unitary representation of $G$, then $\pi$ can be written as an integral of irreducible, unitary representations of $G$ against a positive measure $\mu$ on the unitary dual, $\widehat{G}$. The measure $\mu$ is unique up to an equivalence relation; in particular, the support of $\mu$ is a closed subset of $\widehat{G}$ in the Fell topology that depends only on $\pi$ and not on the choice of representative $\mu$ (see for instance Chapter 8 of \cite{Di} for this theory). The \emph{support} of $\pi$, denoted $\operatorname{supp}\pi$, is defined to be the support of $\mu$. Roughly, $\operatorname{supp}\pi$ is the set of irreducible, unitary representations contained in the decomposition of $\pi$ into irreducibles.

\begin{theorem} \label{maintheorem} Let $\pi$ be a unitary representation of a real, reductive algebraic group $G$, and suppose $\operatorname{supp}\pi\subset \widehat{G}_s$. Then 
$$\operatorname{WF}(\pi)\subset \operatorname{SS}(\pi)\subset i\mathfrak{g}_s^*.$$
\end{theorem}

In short, the wave front set of an integral of singular, irreducible representations is contained in the singular set. This Theorem will be proved in Sections 2 and 3.
\bigskip

When Theorem \ref{maintheorem} is combined with Theorem 1.1 of \cite{HHO}, Theorem 1.2 of \cite{HHO}, and Proposition 1.5 of \cite{How}, we obtain numerous new applications to induction and restriction problems. We begin with induction problems. If $V$ is a finite dimensional vector space, a subset $\mathcal{C}\subset V$ is a \emph{cone} in $V$ if $t\mathcal{C}=\mathcal{C}$ for every positive real number $t\in \mathbb{R}_{>0}$. If $S\subset V$ is any subset, define the \emph{asymptotic cone} of $S$ in $V$ to be
$$\operatorname{AC}(S)=\{\xi\in V|\ \xi\in \mathcal{C}\ \text{an\ open\ cone}\implies \mathcal{C}\cap S\ \text{unbounded}\}\cup \{0\}.$$
The asymptotic cone of any subset $S\subset V$ is always a closed cone in $V$.
If $\sigma\in \widehat{G}_{\text{temp}}^{\text{\ }\prime}$ is an irreducible, tempered representation of $G$ with regular infinitesimal character, then, following Duflo and Rossmann \cite{Du70}, \cite{Ro78}, we associate to $\sigma$ a single coadjoint orbit $\mathcal{O}_{\sigma}\subset i\mathfrak{g}^*$.

\begin{corollary} \label{inducedcor} Suppose $G$ is a real, reductive algebraic group, suppose $H\subset G$ is a closed subgroup, and suppose $(\tau,W)$ is a unitary representation of $H$. Let $\mathfrak{g}$ (resp. $\mathfrak{h}$) denote the Lie algebra of $G$ (resp. $H$), let $q:i\mathfrak{g}^*\rightarrow i\mathfrak{h}^*$ be the natural projection, and let $(\mathfrak{g}^*)'$ denote the set of regular, semisimple elements in $\mathfrak{g}^*$. Then  
$$\operatorname{AC}\left(\bigcup_{\substack{\sigma\in \operatorname{supp}\operatorname{Ind}_H^G\tau\\ \sigma\in \widehat{G}_{\text{temp}}^{\text{\ }\prime}}}\mathcal{O}_{\sigma}\right)\supset \overline{\operatorname{Ad}^*(G)\cdot q^{-1}(\operatorname{SS}(\tau))}\cap i(\mathfrak{g}^*)'.$$
\end{corollary} 

Since $\operatorname{WF}(\tau)\subset \operatorname{SS}(\tau)$, the same statement holds with $\operatorname{SS}(\tau)$ replaced by $\operatorname{WF}(\tau)$. This statement shows the existence of families of tempered representations in many induction problems. Let us look at a special case. Let $X$ be a homogeneous space for $G$ with a non-zero invariant density. Whenever $x\in X$ is a point, we have a surjection
$G\rightarrow X$ by $g\mapsto g\cdot x$. These maps gives rise to surjective maps on tangent spaces which pull back to a family of injective maps on cotangent spaces
$$iT_x^*X\hookrightarrow iT^*_eG\cong i\mathfrak{g}^*$$
for every $x\in X$. We will use these maps to identify $iT_x^*X$ as a subset of $i\mathfrak{g}^*$ for every $x\in X$. 

\begin{corollary} \label{L^2cor} Let $G$ be a real, reductive algebraic group, and let $X$ be a homogeneous space for $G$ with a non-zero invariant density. Then 
$$\operatorname{AC}\left(\bigcup_{\substack{\sigma\in \operatorname{supp}L^2(X)\\ \sigma\in \widehat{G}_{\text{temp}}^{\text{\ }\prime}}}\mathcal{O}_{\sigma}\right)\supset \overline{\bigcup_{x\in X} iT_x^*X}\cap i(\mathfrak{g}^*)'.$$
\end{corollary}

Suppose $G=\operatorname{Sp}(2n,\mathbb{R})$ is the symplectic group of invertible linear transformations preserving the standard symplectic form on $\mathbb{R}^{2n}$, and embed $\operatorname{GL}(n,\mathbb{R})$ into $\operatorname{Sp}(2n,\mathbb{R})$ by
\[A\mapsto \left(\begin{matrix} A & 0\\ 0 & A^{-1}\end{matrix}\right).\]
Suppose $\mathfrak{p}=(p_1,\ldots,p_k)$ and $\mathfrak{q}=(q_1,\ldots,q_l)$ are sequences of nonnegative integers with $\sum p_i+\sum q_j\leq n$, and let 
$$G_{p,q}=\prod_i \operatorname{GL}(p_i,\mathbb{R})\times \prod_j \operatorname{GL}(q_j,\mathbb{Z})$$
be the corresponding diagonal subgroup of $\operatorname{GL}(n,\mathbb{R})$ and $\operatorname{Sp}(2n,\mathbb{R})$. Let
$$X_{\mathfrak{p},\mathfrak{q}}=\operatorname{Sp}(2n,\mathbb{R})/G_{\mathfrak{p},\mathfrak{q}}$$
be the corresponding homogeneous space. Then (see Section 4) a calculation shows that the right hand side of Corollary \ref{L^2cor} contains all of the regular, elliptic elements in $i\mathfrak{g}^*$ when $X=X_{\mathfrak{p},\mathfrak{q}}$. Recall that the Harish-Chandra discrete series of $G=\operatorname{Sp}(2n,\mathbb{R})$ fall into a finite number of families depending on conjugacy classes of Weyl chambers (see \cite{HC66} for the original reference or page 310 of \cite{Kn86} for an exposition). If $\mathcal{F}$ is one of these families, then we deduce that there are infinitely many distinct $\sigma\in \mathcal{F}$ for which
$$\operatorname{Hom}_G(\sigma,L^2(X_{\mathfrak{p},\mathfrak{q}}))\neq \{0\}$$
for every $\mathfrak{p}$ and $\mathfrak{q}$. In particular, $L^2(X_{\mathfrak{p},\mathfrak{q}})$ contains infinitely many distinct Harish-Chandra discrete series of $G=\operatorname{Sp}(2n,\mathbb{R})$ for every $\mathfrak{p}$ and $\mathfrak{q}$. In Section 4 we will show how to deduce these results from Corollary \ref{L^2cor}, and we will give additional families of examples.

Next, we turn our attention to restriction problems.

\begin{corollary} \label{restcor} Suppose $\pi$ is a unitary representation of a Lie group $G$, and suppose $H\subset G$ is a closed, real, reductive algebraic subgroup. Let $\mathfrak{g}$ (resp. $\mathfrak{h}$) denote the Lie algebra of $G$ (resp. $H$), let
$$q:i\mathfrak{g}^*\rightarrow i\mathfrak{h}^*$$
denote the pullback of the inclusion of Lie algebras, and let $i(\mathfrak{h}^*)'$ denote the set of regular, semisimple elements in $i\mathfrak{h}^*$. Then 
$$\operatorname{AC}\left(\bigcup_{\substack{\sigma\in \operatorname{supp}\pi|_H\\ \sigma\in \widehat{H}_{\text{temp}}^{\text{\ }\prime}}}\mathcal{O}_{\sigma}\right)\supset \overline{q(\operatorname{SS}(\pi))}\cap (\mathfrak{h}^*)'.$$
\end{corollary}

Again, one can replace $\operatorname{SS}(\pi)$ by the smaller set $\operatorname{WF}(\pi)$ in the above Corollary. Let us consider an example. Suppose $G=\operatorname{GL}(2n,\mathbb{R})$, $H=\operatorname{SO}(n,n)$, and $\pi$ is a Stein complementary series representation of $G=\operatorname{GL}(2n,\mathbb{R})$ (see \cite{St67} for the original paper introducing Stein complementary series; see \cite{Vo86} to understand how they fit into the full unitary dual of $\operatorname{GL}(2n,\mathbb{R})$). We show in Section 5 that in this case
$$\overline{q(\operatorname{WF}(\pi))}= i\mathfrak{h}^*.$$ 
This means that the tempered part of the decomposition of $\pi|_H$ into irreducibles is ``asymptotically dense'' in $\widehat{H}_{\text{temp}}$. In particular, if $n$ is even, then $\operatorname{SO}(n,n)$ has Harish-Chandra discrete series representations and we deduce that there are infinitely many distinct Harish-Chandra discrete series representations $\sigma$ of $\operatorname{SO}(n,n)$ for which
$$\operatorname{Hom}_{\operatorname{SO}(n,n)}(\sigma,\pi|_H)\neq \{0\}.$$

We finish this introduction with two remarks. First, Corollary \ref{inducedcor}, Corollary \ref{L^2cor}, and Corollary \ref{restcor} are all in the spirit of the orbit method. For an introduction to this philosophy, the author recommends Kirillov's book \cite{Kir} and Vergne's expository article \cite{Ve83}.

The second remark is that all of the results in this piece should hold when $G$ is a reductive Lie group of Harish-Chandra class. The author states and proves them only for real, reductive algebraic groups only because many of the key results on irreducible, unitary representations have only been written down in this degree of generality.

\section{On Characters and Contours}

Let $G$ be a real, reductive algebraic group. We borrow some notation from Section 6 of \cite{ALTV}. A \emph{Langlands parameter} for $G$ is a triple $\Gamma=(H,\gamma,R_{i\mathbb{R}}^+)$ where $H\subset G$ is a Cartan subgroup with Lie algebra $\mathfrak{h}$, $\gamma$ is a level one character of the $\rho_{\text{abs}}$ double cover of $H$ (see Section 5 of \cite{ALTV} for an explanation), and $R_{i\mathbb{R}}^+$ is a choice of positive roots among the set of imaginary roots for $H$ in $\mathfrak{g}$ for which $d\gamma\in \mathfrak{h}_{\mathbb{C}}^*$ is weakly dominant. This triple must satisfy a couple of other technical assumptions (see Theorem 6.1 of \cite{ALTV}). 

To each Langlands parameter, one associates a standard representation $I(\Gamma)$, which is a finite length, possibly nonunitary representation of $G$ on a Hilbert space and a Langlands quotient $J(\Gamma)$, which is an irreducible, possibly nonunitary representation of $G$ on a Hilbert space. If $\Gamma$ and $\Gamma_1$ are conjugate under the action of $G$, then $I(\Gamma)$ (resp. $J(\Gamma)$) is isomorphic to $I(\Gamma_1)$ (resp. $J(\Gamma_1)$) as a representation of $G$. In addition, every irreducible, nonunitary representation can be realized as $J(\Gamma)$ for some $\Gamma$, and the $\Gamma$ in this realization is unique up to conjugacy (see \cite{La} for the original reference and Section 6 of \cite{ALTV}, Chapter 6 of \cite{Vo81}, and Chapters 8 and 14 of \cite{Kn86} for expositions). In order to construct $I(\Gamma)$ from $\Gamma$, we must break up $\Gamma$ into a discrete piece and a continuous piece.

Following definition 6.5 of \cite{ALTV}, we say that $\Lambda=(T,\lambda,R_{i\mathbb{R}}^+)$ is a \emph{discrete Langlands parameter} for $G$ if $T$ is the maximal compact subgroup of a Cartan subgroup in $H\subset G$ with Lie algebra $\mathfrak{t}$, $\lambda$ is a level one character of the $\rho_{i\mathbb{R}}$ double cover of the $T$ (see Section 5 of \cite{ALTV} for a definition), and $R_{i\mathbb{R}}^+$ is a choice of positive roots for the set of imaginary roots of $H$ in $\mathfrak{g}$ for which $d\lambda\in i\mathfrak{t}^*$ is weakly dominant. This triple must in addition satisfy a couple of technical assumptions (see Definition 6.5 of \cite{ALTV}). Similarly, a \emph{continuous parameter} for $\Lambda$ is a pair $(A,\nu)$ where $TA=H$ is a Cartan subgroup in $G$ and $\nu$ is a (possibly nonunitary) character of $A$. In addition, $(A,\nu)$ must satisfy a technical assumption (see Definition 6.5 of \cite{ALTV}). 

If $\Gamma$ is a Langlands parameter for $G$, then we may decompose $\Gamma$ into a discrete Langlands parameter $\Lambda$ and a continuous parameter $(A,\nu)$. If $MA$ is the Langlands decomposition of $Z_G(A)$, the centralizer of $A$ in $G$, then we may associate a limit of Harish-Chandra discrete series representation of $M$, $D(\Lambda)$, to $\Lambda$ (see Section 9 of \cite{ALTV} or page 460 of \cite{Kn86}). One can then tensor with the (possibly nonunitary) representation $\nu$ of $A$, and choose a parabolic $P=MAN$ making the real part of $\nu$ weakly dominant for the weights of $\mathfrak{a}$, the Lie algebra of $A$, in $\mathfrak{n}$, the Lie algebra of $N$. Finally, we extend $D(\Lambda)\otimes \nu$ trivially on $N$ to $P$ and we induce from $P$ to $G$. This induced representation is $I(\Gamma)$ (see Section 6 of \cite{ALTV}).

The representation $J(\Gamma)$ is the unique irreducible quotient of the representation $I(\Gamma)$. If $\Gamma=(\Lambda,\nu)$, then we will write $I(\Lambda,\nu)$ (resp. $J(\Lambda,\nu)$) for $I(\Gamma)$ (resp. $J(\Gamma)$). 
\bigskip

Let $G_{\mathbb{C}}$ denote the complexification of the real, reductive algebraic group $G$. If $\Gamma$ is a Langlands parameter, then we write $\mathcal{O}_{\Gamma}^{G_{\mathbb{C}}}$ for the complex coadjoint orbit in $\mathfrak{g}_{\mathbb{C}}^*$ through $d\gamma\in \mathfrak{h}_{\mathbb{C}}^*$. Recall that $\mathcal{O}_{\Gamma}^{G_{\mathbb{C}}}$ is the infinitesimal character of both $I(\Gamma)$ and $J(\Gamma)$. We note that for a fixed Langlands parameter $\Gamma$, there are finitely many Langlands parameters $\Gamma_1$ (up to $G$-conjugacy) for which $\mathcal{O}_{\Gamma}^{G_{\mathbb{C}}}=\mathcal{O}_{\Gamma_1}^{G_{\mathbb{C}}}$.

Let $\Theta_{I(\Gamma)}$ denote the Harish-Chandra character of $I(\Gamma)$, and let $\Theta(\Gamma)$ denote the Harish-Chandra character of $J(\Gamma)$ (see \cite{HC54}, \cite{HC56}, \cite{HC65} for the original references and chapter 10 of \cite{Kn86} for an exposition). These are generalized functions on $G$. If $\Gamma$ is a Langlands parameter, then we may write 
$$\Theta(\Gamma)=\sum_{\Gamma_1} M_{\Gamma_1,\Gamma} \Theta_{I(\Gamma_1)}$$
in a unique way as a finite sum with integer coefficients (see \cite{HC54}, \cite{La} for the original references and Chapter 10 of \cite{Kn86}, Section 15 of \cite{ALTV} for expositions). 
Recall $M_{\Gamma_1,\Gamma}\neq 0$ implies $\mathcal{O}_{\Gamma}^{G_{\mathbb{C}}}=\mathcal{O}_{\Gamma_1}^{G_{\mathbb{C}}}$. 
Define the Lie algebra analogue of the character of $I(\Gamma)$ (resp. $J(\Gamma)$) to be
$$\theta_{I(\Gamma)}=j_G^{1/2}\exp^*\Theta_{I(\Gamma)}\ (\text{resp.}\ \theta(\Gamma)=j_G^{1/2}\exp^*\Theta(\Gamma))$$
where $j_G^{1/2}$ is the unique analytic square root of the Jacobian of the exponential map normalized so that $j_G^{1/2}(0)=1$. Then $\theta_{I(\Gamma)}$ and $\theta(\Gamma)$ are invariant eigendistributions on $\mathfrak{g}$ \cite{HC65}, and we obtain
$$\theta(\Gamma)=\sum_{\Gamma_1} M_{\Gamma_1,\Gamma} \theta_{I(\Gamma_1)}.$$

\bigskip
Next, we associate to every standard representation, $I(\Gamma)$, a contour, $X(\Gamma)$, utilizing the work of Duflo and Rossmann \cite{Du70}, \cite{Ro78}, \cite{Ro80}, \cite{Ro84}. We define this contour as follows. First, break $\Gamma=(H,\gamma,R_{i\mathbb{R}}^+)$ into a discrete Langlands parameter $\Lambda=(T,\lambda,R_{i\mathbb{R}}^+)$ and a continuous Langlands parameter $(A,\nu)$. Let $MA=Z_G(A)$ be the Langlands decomposition of the centralizer of $A$ in $G$. Utilizing Rossmann's character formulas \cite{Ro78}, \cite{Ro80}, we associate a finite union of regular, coadjoint $M$ orbits $\mathcal{O}^M_{D(\Lambda)}\subset i\mathfrak{m}^*$ to the discrete Langlands parameter $\Lambda$. Each of these coadjoint orbits will contain $M\cdot d\lambda$ in its closure, and this union will be the single coadjoint orbit $M\cdot d\lambda$ whenever $d\lambda$ is regular. If $P=MAN$ is the parabolic subgroup we used to construct $I(\Gamma)$ (so the real part of $\nu$ is weakly dominant for the weights of $\mathfrak{a}$ in $\mathfrak{n}$) and $\mathfrak{p}$ is the Lie algebra of $P$, then we associate to $I(\Gamma)$ the contour
$$X(\Gamma)=K\cdot (\mathcal{O}_{D(\Lambda)}^M+\nu+i(\mathfrak{g}/\mathfrak{p})^*)$$
in $\mathfrak{g}_{\mathbb{C}}^*$. Here $K\subset G$ is a maximal compact subgroup of $G$, which can be written as the fixed point set of an involution of $G$ that leaves $H$ stable. We note that when $\nu$ is purely imaginary, $I(\Gamma)$ is a tempered representation and $X(\Gamma)$ is simply the union of coadjoint orbits $G\cdot (\mathcal{O}_{D(\Lambda)}^M+\nu)\subset i\mathfrak{g}^*$ associated to $I(\Gamma)$ by Duflo and Rossmann \cite{Du70}, \cite{Ro78}, \cite{Ro80}. 

The contour $X(\Gamma)$ comes with a natural analytic density; we will often abuse notation and write $X(\Gamma)$ for both the contour and the density. To define this density, we first recall the definition of the canonical density on a coadjoint $M$ orbit $\mathcal{O}^M_{\xi}=M\cdot \xi\subset i\mathfrak{m}^*$. There is a natural symplectic $2$ form on $\mathcal{O}^M_{\xi}$ defined by 
$$\omega_{\eta}(\operatorname{ad}_X^*\eta,\operatorname{ad}_Y^*\eta)=\eta([X,Y])$$
for every $\eta\in \mathcal{O}^M_{\xi}$. If $\dim \mathcal{O}^M_{\xi}=2m$, then we may define a top dimensional form on $\mathcal{O}^M_{\xi}$ by $$\frac{\omega^{\wedge m}}{(2\pi)^mm!}.$$
The absolute value of this top dimensional form is the required canonical density on our coadjoint orbit $\mathcal{O}^M_{\xi}$. Now, $\mathcal{O}_{D(\Lambda)}^M$ is a finite union of coadjoint orbits so we may give it the corresponding finite union of canonical densities. Translating, we have a density on $\mathcal{O}_{D(\Lambda)}^M+\nu$. 

Next, we have an analytic map with compact fibers
$$K\times_{K\cap M} (\mathcal{O}^M_{D(\Lambda)}+\nu+i(\mathfrak{g}/\mathfrak{p})^*)\longrightarrow X(\Gamma).$$
Thus, we may give an analytic density on the left and push it forward to an analytic density on $X(\Gamma)$. To do this, we note that we have an isomorphism $\mathfrak{k}/(\mathfrak{k}\cap \mathfrak{m})\cong \mathfrak{g}/\mathfrak{p}$. Thus, we have a natural pairing
$$\mathfrak{k}/(\mathfrak{k}\cap \mathfrak{m})\times i(\mathfrak{g}/\mathfrak{p})^*\rightarrow \mathbb{C}$$ and a corresponding natural pairing of top dimensional alternating tensors on these vector spaces
$$\Lambda^{\text{top}}(\mathfrak{k}/(\mathfrak{k}\cap \mathfrak{m}))^*\times \Lambda^{\text{top}}i(\mathfrak{g}/\mathfrak{p}) \rightarrow \mathbb{C}.$$
Choose top dimensional alternating tensors on these spaces that pair to a square root of $-1$, and call them $\alpha$ and $\beta$. Then extend $\alpha$ to an invariant top dimensional form on $K/(K\cap M)$ and extend $\beta$ to a translation invariant top dimensional form on $i(\mathfrak{g}/\mathfrak{p})^*$. Now, define a density on 
$$K\times_{K\cap M} (\mathcal{O}^M_{D(\Lambda)}+\nu+i(\mathfrak{g}/\mathfrak{p})^*)$$
by tensoring the absolute value of $\alpha$ with the canonical densities on the coadjoint orbits and tensoring again with the absolute value of $\beta$. This density is independent of the above choice of a square root of $-1$ since we took absolute values. Now, push this density forward to an analytic density on $X(\Gamma)$; this is the density we will work with from now on.

Next, to each irreducible representation $J(\Gamma)$, we wish to associate a contour $C(\Gamma)$ defined by
$$C(\Gamma)=\sum_{\Gamma_1} M_{\Gamma_1,\Gamma}X(\Gamma_1).$$
We also denote by $C(\Gamma)$ the associated linear combination of densities.

Write $\mathcal{D}_c^{\infty}(\mathfrak{g})$ for the space of smooth, compactly supported densities on $\mathfrak{g}$. If $\omega\in \mathcal{D}_c^{\infty}(\mathfrak{g})$, then we define the Fourier transform of $\omega$ to be
$$\mathcal{F}[\omega](\xi)=\int_{\mathfrak{g}} e^{\langle \xi,X\rangle} d\omega_{X}.$$
We note that $\mathcal{F}[\omega]$ is a Schwartz function on $i\mathfrak{g}^*$ and, by the Paley-Wiener Theorem, a complex analytic function on $\mathfrak{g}_{\mathbb{C}}^*$. The Fourier transform realizes the key relationship between characters and contours.

\begin{proposition}[Duflo, Rossmann]\label{prop:Duflo_Rossmann} If $\omega$ is a smooth, compactly supported density on $\mathfrak{g}$, then for every Langlands parameter $\Gamma$, we have the equality
$$\langle \theta(\Gamma),\omega\rangle=\langle C(\Gamma),\mathcal{F}[\omega]\rangle.$$
\end{proposition}

This result follows from \cite{Du70}, \cite{Ro78}, \cite{Ro80}, and pages 377-378 of \cite{Ro84}.
\bigskip

We end this section with a Lemma showing that the asymptotics of unions of singular cycles are contained in the singular set. In the next section, we will use this Lemma in the proof of Theorem \ref{maintheorem}. As in the introduction, if $S\subset V$ is a subset of a finite dimensional, real vector space, then the asymptotic cone of $S$ in $V$ is defined by
$$\operatorname{AC}(S)=\{\xi\in V|\ \xi\in \mathcal{C}\ \text{an\ open\ cone}\implies \mathcal{C}\cap S\ \text{unbounded}\}\cup \{0\}.$$ 
Further, if $C(\Gamma)\subset \mathfrak{g}_{\mathbb{C}}^*$ is one of the above cycles, let
$\operatorname{supp}C(\Gamma)$ denote the support of this cycle. 

\begin{lemma} \label{singularbound} We have the inclusion
$$\operatorname{AC}\left(\bigcup_{J(\Gamma)\in \widehat{G}_s} \operatorname{supp}C(\Gamma)\right)\subset i\mathfrak{g}_s^*.$$
\end{lemma}

As defined in the introduction, $\widehat{G}_s=\widehat{G}-\widehat{G}_{\text{temp}}^{\text{\ }\prime}$ is the set of singular, unitary representations of $G$ and $i\mathfrak{g}_s^*=i\mathfrak{g}^*-i(\mathfrak{g}^*)'$ is the set of singular elements in $i\mathfrak{g}^*$.

\begin{proof}  Recall $\widehat{G}$ denotes the set of irreducible, unitary representations of $G$. Define the \emph{extended unitary dual} of $G$ to be  
$$\widehat{G}_{\text{ext}}=\{J(\Gamma_1)|\ \exists\ \Gamma\ \text{s.t.}\ M_{\Gamma_1,\Gamma}\neq 0\ \text{and}\ J(\Gamma)\in \widehat{G}\}.$$ Similarly, recall from the introduction $\widehat{G}_s=\widehat{G}-\widehat{G}_{\text{temp}}^{\text{\ }\prime}$ is the set of singular, unitary representations of $G$ and define the \emph{singular, extended unitary dual} of $G$ to be 
$$\widehat{G}_{\text{ext},s}=\{J(\Gamma_1)|\ \exists\ \Gamma\ \text{s.t.}\ M_{\Gamma_1,\Gamma}\neq 0\ \text{and}\ J(\Gamma)\in \widehat{G}_s\}.$$

Since $C(\Gamma)=\sum_{\Gamma_1}M_{\Gamma_1,\Gamma}X(\Gamma_1)$, to prove the Lemma, it is enough to show
$$\operatorname{AC}\left(\bigcup_{J(\Gamma)\in \widehat{G}_{\text{ext},s}} X(\Gamma)\right)\subset i\mathfrak{g}_s^*.$$  
Next, we require an elementary Lemma.

\begin{lemma} \label{tinylemma} Suppose $\{S_i\}_{i\in I}$ is a finite collection of subsets of a real, finite dimensional vectors space. Then
$$\operatorname{AC}\left(\bigcup_{i\in I} S_i\right)\subset \bigcup_{i\in I} \operatorname{AC}\left(S_i\right).$$
\end{lemma}

We leave the elementary proof to the reader. To use this Lemma in our context, we will break up $\widehat{G}_{\text{ext},s}$ into a finite number of pieces. If $\Gamma=(H,\gamma,R_{i\mathbb{R}}^+)$ is a Langlands parameter, let $H(\Gamma):=H$ denote the associated Cartan subgroup. If $H$ and $H_1$ are Cartan subgroups of $G$, define
$$(\widehat{G}_{\text{ext},s})_{H,H_1}$$
$$=\{J(\Gamma_1)|\ H_1=H(\Gamma_1),\ \exists\ \Gamma\ \text{s.t.}\ M_{\Gamma_1,\Gamma}\neq \{0\}\ \&\ J(\Gamma)\in \widehat{G}_s,\ H=H(\Gamma)\}.$$
Since there are finitely many conjugacy classes of Cartan subgroups of $G$, by Lemma \ref{tinylemma}, it is enough to show
$$\operatorname{AC}\left(\bigcup_{J(\Gamma_1)\in (\widehat{G}_{\text{ext},s})_{H,H_1}} X(\Gamma_1)\right)\subset i\mathfrak{g}_s^*$$
for every pair of Cartan subgroups $H$, $H_1$.

\begin{lemma} [Langlands] \label{realpartlemma} Fix a Cartan subgroup $H\subset G$, decompose $H=TA$, let $\mathfrak{a}$ denote the Lie algebra of $A$, and let $i\mathfrak{a}^*$ denote the set of imaginary valued linear functionals on $\mathfrak{a}$ (or alternately unitary one dimensional characters of $A$). Then there exists a bounded subset $B_H\subset \mathfrak{a}^*$ such that if $\Gamma$ is a Langlands parameter with $H(\Gamma)=H$ and continuous part $(A,\nu)$ for which $J(\Gamma)\in \widehat{G}_{\text{ext}}$, then $$\nu\in i\mathfrak{a}^*+B_H.$$
\end{lemma}

This Lemma follows from results of Langlands \cite{La}, which make heavy use of results of Harish-Chandra \cite{HC66}, \cite{HC84}. An exposition of these results can be found in \cite{Kn86}, in particular, see Theorem 8.47 and Theorem 8.61. 

In passing, we note that more precise bounds on the size of $B_H$ appeared in Proposition 7.18 of \cite{SV}. However, we will not need precise bounds in this paper.

Suppose $\Gamma=(H,\gamma,R_{i\mathbb{R}}^+)$ is a Langlands parameter with discrete part $\Gamma=(T,\lambda,R_{i\mathbb{R}}^+)$ and continuous part $(A,\nu)$. Then we will write $\nu=\operatorname{Re}\nu+\operatorname{Im}\nu$ with $\operatorname{Re}\nu\in \mathfrak{a}^*$ and $\operatorname{Im}\nu\in i\mathfrak{a}^*$. One can write the conclusion of Lemma \ref{realpartlemma} as $\operatorname{Re}\nu\in B_H$.
\bigskip

Now, back to the proof of Lemma \ref{singularbound}. Fix Cartan subgroups $H$ and $H_1$ for which $(\widehat{G}_{\text{ext},s})_{H,H_1}$ is nonempty. Fix $J(\Gamma)\in \widehat{G}$ with $H(\Gamma)=H$ and suppose $J(\Gamma_1)\in (\widehat{G}_{\text{ext},s})_{H,H_1}$ with $M_{\Gamma_1,\Gamma}\neq 0$. Decompose $\Gamma=(H,\gamma,R_{i\mathbb{R}}^+)$ into a discrete Langlands parameter $\Lambda=(T,\lambda,R_{i\mathbb{R}}^+)$ and a continuous parameter $(A,\nu)$. Further, let $Z_G(A)=MA$ be the Langlands decomposition of $Z_G(A)$. Let $\mathfrak{h}$ denote the Lie algebra of $H$, let $\mathfrak{t}$ be the Lie algebra of $T$, let $\mathfrak{a}$ be the Lie algebra of $A$, and note $$d\lambda\in i\mathfrak{t}^*,\ \nu\in \mathfrak{a}_{\mathbb{C}}^*,\ \text{and}\ d\gamma=d\lambda+\nu\in \mathfrak{h}_{\mathbb{C}}^*.$$ 
Similarly, decompose the Langlands parameter $\Gamma_1=(H_1,\gamma_1,(R_{\mathbb{R}}^+)_1)$ into a discrete Langlands parameter $\Lambda_1=(T_1,\lambda_1,(R_{i\mathbb{R}}^+)_1)$ and a continuous parameter $(A_1,\nu_1)$. Let $Z_G(A_1)=M_1A_1$ be the Langlands decomposition of $Z_G(A_1)$. Let $\mathfrak{h}_1$ denote the Lie algebra of $H_1$, let $\mathfrak{t}_1$ denote the Lie algebra of $T_1$, let $\mathfrak{a}_1$ denote the Lie algebra of $A_1$, and note $$d\lambda_1\in i\mathfrak{t}_1^*,\ \nu_1\in (\mathfrak{a}_1)_{\mathbb{C}}^*,\ \text{and}\ d\gamma_1=d\lambda_1+\nu_1\in (\mathfrak{h}_1)_{\mathbb{C}}^*.$$ 

Let us look back at the contour $$X(\Gamma_1)=K\cdot (\mathcal{O}_{D(\Lambda_1)}^{M_1}+\nu_1+i(\mathfrak{g}/\mathfrak{p}_1)^*))\subset K\cdot (\mathcal{O}^{M_1}_{D(\Lambda_1)}+\operatorname{Im}\nu_1+i(\mathfrak{g}/\mathfrak{p}_1)^*)+K\cdot \operatorname{Re}\nu_1$$
and define
$$\widetilde{X}(\Gamma_1)=K\cdot (\mathcal{O}_{D(\Lambda_1)}^{M_1}+\operatorname{Im}\nu_1+i(\mathfrak{g}/\mathfrak{p}_1)^*).$$
Fix Cartan subgroups $H$ and $H_1$ for which $(\widehat{G}_{\text{ext},s})_{H,H_1}$ is nonempty. By Lemma \ref{realpartlemma}, $K\cdot \operatorname{Re}\nu_1\in K\cdot B_{H_1}$, a bounded set, for all $\Gamma_1$ with $J(\Gamma_1)\in (\widehat{G}_{\text{ext},s})_{H,H_1}$. Thus, we see
$$\operatorname{AC}\left(\bigcup_{J(\Gamma_1)\in (\widehat{G}_{\text{ext},s})_{H,H_1}}X(\Gamma_1)\right)\subset \operatorname{AC}\left(\bigcup_{J(\Gamma_1)\in (\widehat{G}_{\text{ext},s})_{H,H_1}}\widetilde{X}(\Gamma_1)\right).$$
Therefore, to prove Lemma \ref{singularbound}, we need only show 

\begin{equation}\label{eq:singularbound_reduction}
\operatorname{AC}\left(\bigcup_{J(\Gamma_1)\in (\widehat{G}_{\text{ext},s})_{H,H_1}}\widetilde{X}(\Gamma_1)\right)\subset i\mathfrak{g}_s^*
\end{equation}

Now, let $H_{\mathbb{C}}$ (resp. $(H_1)_{\mathbb{C}}$) denote the complexifications of $H$ (resp. $H_1$) in $G_{\mathbb{C}}$, the complexification of $G$. Notice $H_{\mathbb{C}}$ and $(H_1)_{\mathbb{C}}$ are complex Cartan subgroups of $G_{\mathbb{C}}$. In particular, they are conjugate by $g\in G_{\mathbb{C}}$, and we deduce $\operatorname{Ad}_g\mathfrak{h}_{\mathbb{C}}=(\mathfrak{h}_1)_{\mathbb{C}}$. Recall $M_{\Gamma_1,\Gamma}\neq 0$ implies that the representations $I(\Gamma)$ and $I(\Gamma_1)$ have the same infinitesimal character. If $$W_{\mathbb{C}}=N_{G_{\mathbb{C}}}(H_{\mathbb{C}})/H_{\mathbb{C}},\ W_{1,\mathbb{C}}=N_{G_{\mathbb{C}}}((H_1)_{\mathbb{C}})/(H_1)_{\mathbb{C}}$$ are the complex Weyl groups associated to $H$ and $H_1$, then we deduce that there exists $w\in W_{1,\mathbb{C}}$ such that 
$$w\cdot \operatorname{Ad}_{g^{-1}}^*d\gamma=d\gamma_1\in (\mathfrak{h}_1)^*_{\mathbb{C}}.$$

Put a $W_{\mathbb{C}}$ invariant inner product on $\mathfrak{h}_{\mathbb{C}}^*$, use $\operatorname{Ad}_{g^{-1}}^*$ to transfer it to $(\mathfrak{h}_1)_{\mathbb{C}}^*$, and let $|\cdot|$ denote the corresponding norms on $\mathfrak{h}_{\mathbb{C}}^*$ and $(\mathfrak{h}_1)_{\mathbb{C}}^*$. We know from Lemma \ref{realpartlemma} that there exists a constant $d>0$ such that $|\operatorname{Re}\nu|<d$ and $|\operatorname{Re}\nu_1|<d$ since $J(\Gamma), J(\Gamma_1)\in \widehat{G}_{\text{ext}}$. Moreover, since $J(\Gamma)\in \widehat{G}_{s}$, we deduce that $$d\lambda+\operatorname{Im}\nu\in i\mathfrak{h}^*-i(\mathfrak{h}^*)'$$ is singular (see \cite{KnZ} or Theorem 16.6 of \cite{Kn86}). Thus, $$w\cdot \operatorname{Ad}_{g^{-1}}^*(d\lambda+\operatorname{Im}\nu)\in (\mathfrak{h}_1)^*_{\mathbb{C}}-((\mathfrak{h}_1)^*_{\mathbb{C}})'$$ is singular, and 
\begin{equation}\label{eq:bound}
|(d\lambda_1+\operatorname{Im}\nu_1)-w\cdot \operatorname{Ad}_{g^{-1}}^*(d\lambda+\operatorname{Im}\nu)|\leq |\operatorname{Re}\nu_1|+|\operatorname{Re}\nu|<2d
\end{equation}
by two applications of the triangle inequality.
Thus, the distance between the parameter $d\lambda_1+\operatorname{Im}\nu_1$ and the singular set is at most $2d$ for all parameters $\Gamma_1$ with $J(\Gamma_1)\in (\widehat{G}_{\text{ext},s})_{H,H_1}$. To utilize this information in the proof of Lemma \ref{singularbound}, we need another Lemma.

There exists a natural map
$$p\colon \mathfrak{g}_{\mathbb{C}}^*\longrightarrow (\mathfrak{h}_1)_{\mathbb{C}}^*/W_{1,\mathbb{C}}$$
defined by mapping $\xi\in \mathfrak{g}_{\mathbb{C}}^*$ to the Weyl group orbit of points $\eta\in (\mathfrak{h}_1)_{\mathbb{C}}^*$ for which $q(\xi)=q(\eta)$ for every invariant polynomial $q\in \operatorname{Pol}(\mathfrak{g}_{\mathbb{C}}^*)^{G_{\mathbb{C}}}$.

\begin{lemma} \label{technicalcontour} If $\Gamma_1$ is a Langlands parameter with discrete part $\Lambda_1=(T_1,\lambda_1,(R_{i\mathbb{R}}^+)_1)$ and continuous part $(A_1,\nu_1)$, then the contour 
$$\widetilde{X}(\Gamma_1):=K\cdot (\mathcal{O}_{D(\Lambda_1)}^{M_1}+\operatorname{Im}\nu_1+i(\mathfrak{g}/\mathfrak{p}_1)^*)$$
is contained in the fiber $p^{-1}(d\lambda_1+\operatorname{Im}\nu_1)$.
\end{lemma}

In order to verify Lemma \ref{technicalcontour}, we need an additional Lemma.

\begin{lemma}\label{orbit_equality} If $\xi$ is a regular, semisimple element in $i(\mathfrak{h}_1)^*$, then
\[\mathcal{O}_{\xi}^{M_1}+i(\mathfrak{g}/\mathfrak{p}_1)^* = \operatorname{Ad}^*(P_1)\cdot \xi.\]
\end{lemma}

\begin{proof}

Let $P_1=M_1A_1N_1$ be the Langlands decomposition of $P_1$. Since $M_1A_1$ preserves $i(\mathfrak{g}/\mathfrak{p}_1)^*$ and acts by automorphisms, it is enough to verify 
\begin{equation}\label{eq:orbit_equality}
\operatorname{Ad}^*(N_1)\cdot \xi=\xi+i(\mathfrak{g}/\mathfrak{p}_1)^*.
\end{equation}
To check this, choose $\mathfrak{a}_1\subset \widetilde{\mathfrak{a}_1}$ where $\widetilde{\mathfrak{a}_1}\subset \mathfrak{g}$ is the split part of a maximally split Cartan subalgebra $\widetilde{\mathfrak{h}_1}\subset \mathfrak{g}$. Let $\Sigma\subset i\widetilde{\mathfrak{a}_1}^*\setminus \{0\}$ denote the collection of restricted roots of $\mathfrak{g}$ with respect to $\widetilde{\mathfrak{a}_1}$.
By Proposition 7.76 of \cite{Kna05}, there exists a choice of positive restricted roots $\Sigma^+\subset \Sigma$ with simple roots $\Pi\subset \Sigma^+$ and a subset $\Pi'\subset \Pi$ for which 
\[\mathfrak{p}_1=Z_{\mathfrak{g}}(\widetilde{\mathfrak{a}_1})\oplus \sum_{\alpha\in \Sigma'}\mathfrak{g}_{\alpha}\]
where $\Sigma'=\Sigma^+\cup \{\alpha\in \Sigma\mid \alpha\in \operatorname{Span}(\Pi')\}$. Fix a nondegenerate, $G$-invariant bilinear form $B$ on $\mathfrak{g}$, use it to identify $i\mathfrak{g}\simeq i\mathfrak{g}^*$, and suppose $\xi$ corresponds to $X\in i\mathfrak{h}_1$. Under this isomorphism $i(\mathfrak{g}/\mathfrak{p}_1)^*$ corresponds to $i\mathfrak{n}_1$ where 
\[\mathfrak{n}_1=\bigoplus_{\alpha\in \Sigma^+\cap \operatorname{Span}(\Pi')} \mathfrak{g}_{\alpha}.\]
In order to verify (\ref{eq:orbit_equality}), it is enough to check that for every $Y\in i\mathfrak{n}_1$, there exists $n\in N_1=\exp(\mathfrak{n}_1)$ such that
\begin{equation}\label{eq:Kostant_equality}
\operatorname{Ad}(n)X=X+Y.
\end{equation}
Note that every element $\alpha\in \Sigma^+\cap \operatorname{Span}(\Pi')$ can be written uniquely as a sum of simple roots in $\Pi'$; the number of simple roots in this sum is called the length of $\alpha$ and is denoted by $l(\alpha)$. Define
\[(\mathfrak{n}_1)_m=\bigoplus_{\substack{\alpha\in \Sigma^+\cap \operatorname{Span}(\Pi')\\ l(\alpha)> m}} \mathfrak{g}_{\alpha}.\]
We will show by induction that for every $m\in \mathbb{Z}_{\geq 0}$ there exists $n_m\in \exp(\mathfrak{n}_1)$ such that 
\begin{equation}\label{eq:Kostant_equality_m}
\operatorname{Ad}(n_m)X-(X+Y)\in i(\mathfrak{n}_1)_{\mathfrak{m}}.
\end{equation}
Since $(\mathfrak{n}_1)_m=0$ for sufficiently large $m$, this will verify (\ref{eq:Kostant_equality}). The base case $m=0$ follows with $n_0=e=\exp(0)$. Assume (\ref{eq:Kostant_equality_m}) for a fixed non-negative integer $m$. For every restricted root $\alpha$ with $\mathfrak{g}_{\alpha}\subset \mathfrak{n}_1$ and $l(\alpha)=m+1$, let $X_{\alpha}$ denote the component of $\operatorname{Ad}(n_m)X-(X+Y)$ in the root space $i\mathfrak{g}_{\alpha}$. Observe $\alpha(X)\neq 0$ since $\xi$, and therefore $X$, was assumed to be regular. Hence, we may define
\[Z_{m+1}=\sum_{\substack{\alpha\in \Sigma^+\cap \operatorname{Span}(\Pi')\\ l(\alpha)=m+1}} \frac{-X_{\alpha}}{\alpha(X)}.\]
Then
\[\operatorname{Ad}(\exp(Z_{m+1}))\operatorname{Ad}(n_m)X=\operatorname{Ad}(\exp(Z_{m+1}))X+\operatorname{Ad}(\exp(Z_{m+1}))Y+W\]
where 
\[W=\operatorname{Ad}(\exp(Z_{m+1}))(\operatorname{Ad}(n_m)X-(X+Y))\in i(\mathfrak{n}_1)_{m+1}.\]
Moreover, 
\[\operatorname{Ad}(\exp(Z_{m+1}))X=X+\operatorname{ad}(Z_{m+1})X+W'\]
where $W'\in i(\mathfrak{n}_1)_{m+1}$. And
\[\operatorname{ad}(Z_{m+1})X=\sum_{\substack{\alpha\in \Sigma^+\cap \operatorname{Span}(\Pi')\\ l(\alpha)=m+1}} \frac{-[X_{\alpha},X]}{\alpha(X)}=\sum_{\substack{\alpha\in \Sigma^+\cap \operatorname{Span}(\Pi')\\ l(\alpha)=m+1}} X_{\alpha}.\]
Further,
\[\operatorname{Ad}(\exp(Z_{m+1}))Y=Y+W''\]
with $W''\in i(\mathfrak{n}_1)_{m+1}$. Putting $n_{m+1}=\exp(Z_{m+1})n_m$ and combining the above expressions, we obtain
\[\operatorname{Ad}(n_{m+1})X-(X+Y)\in i(\mathfrak{n}_1)_{m+1}.\]
Therefore, we have verified (\ref{eq:Kostant_equality_m}) for the integer $m+1$. The Lemma follows.
\end{proof}

Next, we prove Lemma \ref{technicalcontour}.

\begin{proof} First, suppose $\xi=d\lambda_1+\operatorname{Im}\nu\in i(\mathfrak{h}_1)^*$ is a regular, semisimple element. Since the fibers of $p$ are $G$ invariant, Lemma \ref{orbit_equality} implies 
\begin{align*}
\widetilde{X}(\Gamma_1)&=\operatorname{Ad}^*(K)\cdot(\mathcal{O}^{M_1}_{d\lambda_1+\operatorname{Im}\nu_1}+i(\mathfrak{g}/\mathfrak{p}_1)^*)\\
&=\operatorname{Ad}^*(K)\operatorname{Ad}^*(P_1)\cdot (d\lambda_1+\operatorname{Im}\nu_1)\\
&\subset p^{-1}(d\lambda_1+\operatorname{Im}\nu_1).
\end{align*}
The Lemma has been proven in the case when $d\lambda_1+\operatorname{Im}\nu_1$ is regular. 

Next, assume $d\lambda_1+\operatorname{Im}\nu_1$ is singular, fix $\eta\in \mathcal{O}_{D(\Lambda_1)}^{M_1}$, and decompose $\eta=\eta_s+\eta_n$ via the Jordan decomposition into semisimple and nilpotent parts. Note $\eta_s=d\lambda_1$, define $\mathfrak{l}:=Z_{\mathfrak{m}_1}(\eta_s)$ and $L:=Z_{M_1}(\eta_s)$, fix a nondegenerate $L$-invariant form $B$ on $\mathfrak{l}$ and use it to identify $\mathfrak{l}\simeq \mathfrak{l}^*$. Then, after possibly replacing $\eta_n$ by a conjugate under $L$, we may find an $\mathfrak{sl}_2$-triple $\{h,e,f\}$ in $\mathfrak{l}$ for which $ie$ corresponds to $\eta_n$ under $B$ and $e-f\in \mathfrak{h}_1'$ (see part (c) of Lemma B of \cite{Ro82}). As remarked on page 223 of \cite{Ro82}, one observes 
\[e=\lim_{t\rightarrow \infty} e^{-2t}\operatorname{Ad}(\exp(th))(e-f).\]
Suppose $i(e-f)$ corresponds to the element $\zeta\in i\mathfrak{h}_1^*$ under $B$, and define $\zeta_t=e^{-2t}\zeta$. Then, for any $\eta'\in i(\mathfrak{g}/\mathfrak{p}_1)^*$, we have
\[\eta+\operatorname{Im}\nu_1+\eta'=\lim_{t\rightarrow \infty} \operatorname{Ad}^*(\exp(th))(\zeta_t)+d\lambda_1+\operatorname{Im}\nu_1+\eta'.\]
Putting $m_t=\exp(th)\in M_1$ and noting that $\operatorname{Ad}^*(m_t)$ stabilizes $d\lambda_1$, $\operatorname{Im}\nu_1$ and preserves $i(\mathfrak{g}/\mathfrak{p}_1)^*$, one obtains
\[\eta+\operatorname{Im}\nu_1+\eta'=\lim_{t\rightarrow \infty} \operatorname{Ad}^*(m_t)\cdot (\zeta_t+ d\lambda_1+\operatorname{Im}\nu_1+\eta'')\]
with $\eta''\in i(\mathfrak{g}/\mathfrak{p}_1)^*$. If $q\in \operatorname{Pol}(\mathfrak{g}_{\mathbb{C}}^*)^{G_{\mathbb{C}}}$, then by continuity and invariance, 
\begin{align*}
q(\eta+\operatorname{Im}\nu_1+\eta')&=q\left(\lim_{t\rightarrow \infty} \operatorname{Ad}^*(m_t)\cdot (\zeta_t+ d\lambda_1+\operatorname{Im}\nu_1+\eta'')\right)\\
&=\lim_{t\rightarrow \infty} q\left(\zeta_t+ d\lambda_1+\operatorname{Im}\nu_1+\eta''\right).
\end{align*}
Next, since $\zeta_t+d\lambda_1+\operatorname{Im}\nu_1$ is regular and semisimple, by Lemma \ref{orbit_equality}, we deduce that $\zeta_t+d\lambda_1+\operatorname{Im}\nu_1$ and $\zeta_t+d\lambda_1+\operatorname{Im}\nu_1+\eta''$ lie in the same coadjoint $G$-orbit and 
\[q(\zeta_t+d\lambda_1+\operatorname{Im}\nu_1+\eta'')=q(\zeta_t+d\lambda_1+\operatorname{Im}\nu_1).\]
Applying continuity of $q$ one more time, we have
\[q(\eta+\operatorname{Im}\nu_1+\eta'')=q(d\lambda_1+\operatorname{Im}\nu_1).\]
Lemma \ref{technicalcontour} follows.
\end{proof}
\bigskip

Now, we will finally finish off the proof of Lemma \ref{singularbound}. Suppose $\xi\in i(\mathfrak{g}^*)'$. By (\ref{eq:singularbound_reduction}) we must show 
\[\xi\notin\operatorname{AC}\left(\bigcup_{J(\Gamma_1)\in (\widehat{G}_{\text{ext},s})_{H,H_1}}\widetilde{X}(\Gamma_1)\right)\]
where $\widetilde{X}(\Gamma_1)=K\cdot (\mathcal{O}^{M_1}_{D(\Lambda_1)}+\operatorname{Im}\nu_1+i(\mathfrak{g}/\mathfrak{p}_1)^*)$. Taking the union separately over $J(\Gamma_1)$ with $d\lambda_1+\operatorname{Im}\nu_1$ singular and $J(\Gamma_1)$ with $d\lambda_1+\operatorname{Im}\nu_1$ regular and applying Lemma \ref{tinylemma}, it is enough to show 
\begin{equation}\label{eq:singularcase}
\xi\notin\operatorname{AC}\left(\bigcup_{\substack{J(\Gamma_1)\in (\widehat{G}_{\text{ext},s})_{H,H_1}\\ d\lambda_1+\operatorname{Im}\nu_1\ \text{singular}}}\widetilde{X}(\Gamma_1)\right)
\end{equation}
and 
\begin{equation}\label{eq:regularcase}
\xi\notin\operatorname{AC}\left(\bigcup_{\substack{J(\Gamma_1)\in (\widehat{G}_{\text{ext},s})_{H,H_1}\\ d\lambda_1+\operatorname{Im}\nu_1\ \text{regular}}}\widetilde{X}(\Gamma_1)\right).
\end{equation}
We begin by checking (\ref{eq:singularcase}). If $d\lambda_1+\operatorname{Im}\nu_1$ is singular, by Lemma \ref{technicalcontour}\[\widetilde{X}(\Gamma_1)\subset p^{-1}(d\lambda_1+\operatorname{Im}\nu_1)\subset i\mathfrak{g}^*_s.\]
Therefore, the union in (\ref{eq:singularcase}) is entirely contained in $i\mathfrak{g}_s^*$. Since the asymptotic cone of any subset of the singular set is contained in the singular set and $\xi\in i(\mathfrak{g}^*)'$ is regular, (\ref{eq:singularcase}) follows.

To check (\ref{eq:regularcase}), observe Lemma \ref{orbit_equality} implies
\[\widetilde{X}(\Gamma_1)=K\cdot (\mathcal{O}^{M_1}_{D(\Lambda_1)}+\operatorname{Im}\nu_1+i(\mathfrak{g}/\mathfrak{p}_1)^*)\subset \mathcal{O}_{d\lambda_1+\operatorname{Im}\nu_1}^G:=\operatorname{Ad}^*(G)\cdot (d\lambda_1+\operatorname{Im}\nu_1)\]
whenever $d\lambda_1+\operatorname{Im}\nu_1$ is regular. In particular, it is enough to check
\begin{equation}\label{eq:regularcase_invariant}
\xi\notin\operatorname{AC}\left(\bigcup_{\substack{J(\Gamma_1)\in (\widehat{G}_{\text{ext},s})_{H,H_1}\\ d\lambda_1+\operatorname{Im}\nu_1\ \text{regular}}}\mathcal{O}_{d\lambda_1+\operatorname{Im}\nu_1}^G\right)
\end{equation}
and the set on the right is $G$-invariant. If $\xi$ is not conjugate to an element of $\mathfrak{h}_1^*$, then choose a Cartan $\xi\in \mathfrak{h}_2^*$ and put 
\[\mathcal{C}:=\operatorname{Ad}^*(G)\cdot (\mathfrak{h}_2^*)'.\]
Then $\xi\in \mathcal{C}\subset \mathfrak{g}^*$ is an open cone that does not intersect 
\[\mathcal{O}_{d\lambda_1+\operatorname{Im}\nu_1}^G\subset \operatorname{Ad}^*(G)\cdot (d\lambda_1+\operatorname{Im}\nu_1)\]
if $J(\Gamma_1)\in (\widehat{G}_{\text{ext},s})_{H,H_1}$. Hence, (\ref{eq:regularcase_invariant}) follows. Therefore, we may assume without loss of generality that $\xi$ is conjugate to an element in $i\mathfrak{h}_1^*$. In fact, since the right hand side of (\ref{eq:regularcase_invariant}) is $G$-invariant, we may assume $\xi\in i\mathfrak{h}_1^*$ without loss of generality.

Recall we fixed a $W_{1,\mathbb{C}}$-invariant norm $|\cdot|$ on $(\mathfrak{h}_1)_{\mathbb{C}}^*$, and we have a corresponding invariant distance function. Note there exists a constant $c>0$ such that
\[\operatorname{dist}(\xi,i\mathfrak{h}_s^*)>c\]
and deduce
\[\operatorname{dist}(t\xi,i\mathfrak{h}_s^*)>tc\]
for all $t>0$. Next, for $\epsilon>0$, define the open cone
\[\mathcal{C}_{\epsilon}(\xi)=\left\{\eta\in i\mathfrak{h}_1^*\setminus \{0\}\mid \left|\frac{\eta}{|\eta|}-\frac{\xi}{|\xi|}\right|<\epsilon\right\} \]
or equivalently 
\[\mathcal{C}_{\epsilon}(\xi)=\left\{\eta\in i\mathfrak{h}_1^*\mid \left|\eta-\frac{|\eta|}{|\xi|}\xi\right|<\epsilon|\eta|\right\}\setminus \{0\}.\]
For $\eta\in \mathcal{C}_{\epsilon}(\xi)$ the triangle inequality yields
\begin{align*}
\operatorname{dist}(\eta,i(\mathfrak{h}_1)_s^*)&\geq \operatorname{dist}\left(\frac{|\eta|}{|\xi|}\xi,i(\mathfrak{h}_1)_s^*\right)-\operatorname{dist}\left(\frac{|\eta|}{|\xi|}\xi,\eta\right)\\
&\geq \frac{|\eta|}{|\xi|}c-\epsilon|\eta|\\
&=|\eta|\left(\frac{c}{|\xi|}-\epsilon\right).
\end{align*}
Now, recall from (\ref{eq:bound}) that 
\[\operatorname{dist}(d\lambda_1+\operatorname{Im}\nu_1,i\mathfrak{h}^*_s)\leq 2d\]
for a fixed constant $d$ whenever $J(\Gamma_1)\in (\widehat{G}_{\text{ext},s})_{H,H_1}$. Combined with the previous inequality, it implies that so long as $\epsilon<\frac{c}{|\xi|}$, every $d\lambda_1+\operatorname{Im}\nu_1\in \mathcal{C}_{\epsilon}(\xi)$ with $J(\Gamma_1)\in (\widehat{G}_{\text{ext},s})_{H,H_1}$ and $d\lambda_1+\operatorname{Im}\nu_1$ regular satisfies 
\[|d\lambda_1+\operatorname{Im}\nu_1|< r(\xi)\]
for some fixed constant $r(\xi)>0$. 

Now, there are finitely many elements of the form $\operatorname{Ad}^*(g)\cdot \xi\in i\mathfrak{h}_1^*$
with $g\in G$ and every such element can be written in the form $w\cdot \xi$ for $w\in W_{1,\mathbb{C}}$; we call $W_{1,\mathbb{R}}$ the collection of such $w$. Therefore, we may form the cone $\mathcal{C}_{\epsilon}(w\cdot \xi)$ for each $w\in W_{1,\mathbb{R}}$, and we may put
\[\mathcal{C}_{\epsilon}:=\bigcup_{w\in W_{1,\mathbb{R}}} \mathcal{C}_{\epsilon}(w\cdot \xi).\]
Notice $\operatorname{Ad}^*(G)\cdot \mathcal{C}_{\epsilon}\cap i\mathfrak{h}_1^*=\mathcal{C}_{\epsilon}$. In addition, repeating the above argument with $\xi$ replaced by $w\cdot \xi$ for all $w\in W_{1,\mathbb{R}}$, one deduces that, after choosing $\epsilon>0$ sufficiently small, there exists $r>0$ for which 
\[|d\lambda_1+\operatorname{Im}\nu_1|< r\]
for every $d\lambda_1+\operatorname{Im}\nu_1\in \mathcal{C}_{\epsilon}$ with $J(\Gamma_1)\in (\widehat{G}_{\text{ext},s})_{H,H_1}$ and $d\lambda_1+\operatorname{Im}\nu_1$ regular. In particular, all such parameters are contained in the ball $B_r(0)\subset i\mathfrak{h}_1^*$.

Finally, choose a precompact, open subset $e\in U\subset G$, and define 
\[\widetilde{\mathcal{C}}_{\epsilon}:=\operatorname{Ad}^*(U)\cdot \mathcal{C}_{\epsilon}.\]
Note $\xi\in \widetilde{\mathcal{C}}_{\epsilon}\subset i\mathfrak{g}^*$ is an open cone. Further, $\eta\in \mathcal{O}^G_{d\lambda_1+\operatorname{Im}\nu_1}$ intersects $\widetilde{\mathcal{C}}_{\epsilon}$ only if $d\lambda_1+\operatorname{Im}\nu_1\in \mathcal{C}_{\epsilon}$ and if true the intersection is 
\[\operatorname{Ad}^*(U)\cdot (d\lambda_1+\operatorname{Im}\nu_1)\subset \operatorname{Ad}^*(U)\cdot B_r(0).\]
The set $\operatorname{Ad}^*(U)\cdot B_r(0)$ is precompact since $U$ is precompact. We conclude 
\[\widetilde{\mathcal{C}}_{\epsilon}\cap \left(\bigcup_{\substack{J(\Gamma_1)\in (\widehat{G}_{\text{ext},s})_{H,H_1}\\ d\lambda_1+\operatorname{Im}\nu_1\ \text{regular}}}\mathcal{O}_{d\lambda_1+\operatorname{Im}\nu_1}^G\right)\]
is bounded; (\ref{eq:regularcase_invariant}) follows. Lemma \ref{singularbound} has been proven.
\end{proof}

\section{On Integrals of Characters and Contours}

Now, suppose $\pi$ is a unitary representation of a real, reductive algebraic group $G$. We may write 
$$\pi\simeq \int_{J(\Gamma)\in \widehat{G}}J(\Gamma)^{\oplus m(\pi,J(\Gamma))} d\mu_{\Gamma}$$
as a direct integral of irreducible, unitary representations. The positive measure $\mu$ on $\widehat{G}$ is only unique up to an equivalence relation. We say $\mu$ and $\mu'$ are equivalent if and only if $\mu$ and $\mu'$ are absolutely continuous with respect to each other (see for instance Chapter 8 of \cite{Di} for this general theory). We can always find a positive measure $\mu'$ that is equivalent to $\mu$ for which $\mu'(\widehat{G})<\infty$. Hence, without loss of generality, we will assume from now on that $\mu$ is a finite positive measure on $\widehat{G}$. In addition, we will often identify $\widehat{G}$ with the set of Langlands parameters $\Gamma$ such that $J(\Gamma)$ is unitary, and we will write $\mu$ both for the measure on $\widehat{G}$ and the measure on the corresponding set of parameters. 

The next step in proving Theorem \ref{maintheorem} is to study integrals of irreducible characters.

\begin{lemma} \label{chint} Suppose $\mu$ is a finite, positive measure on $\widehat{G}$. 
\begin{enumerate}
\item If we form the integral 
\[C(\mu)=\int_{J(\Gamma)\in \widehat{G}} C(\Gamma)d\mu_{\Gamma}\]
then the functional
\[\omega\mapsto \langle C(\mu),\mathcal{F}[\omega] \rangle:=\int_{J(\Gamma)\in \widehat{G}}\langle C(\Gamma),\mathcal{F}[\omega]\rangle d\mu_{\Gamma}\]
defines a distribution on $\mathfrak{g}$ for $\omega\in \mathcal{D}_c^{\infty}(\mathfrak{g})$. By Proposition \ref{prop:Duflo_Rossmann}, we may also write this distribution as
\[\omega\mapsto \langle \theta(\mu), \omega\rangle:=\int_{J(\Gamma)\in \widehat{G}}\langle \theta(\Gamma), \omega\rangle d\mu_{\Gamma}\]
where  
\[\theta(\mu):=\int_{J(\Gamma)\in \widehat{G}} \theta(\Gamma)d\mu_{\Gamma}.\]
\item For every smooth, compactly supported density, $\omega$, on $\mathfrak{g}$, the convolution $$C(\mu)*\mathcal{F}[\omega]$$ 
defines a smooth, polynomially bounded function on $i\mathfrak{g}^*$. 
\end{enumerate}
\end{lemma}

\begin{proof}
To begin the proof, we first show that 
$$\omega\mapsto \langle C(\mu),\mathcal{F}[\omega]\rangle$$
is a distribution on $\mathfrak{g}$. Fix $S\subset \mathfrak{g}$ a compact set, and let $\mathcal{D}_S^{\infty}(\mathfrak{g})$ be the space of smooth densities, $\omega$, supported in $S$. It is enough to show $\omega\mapsto \langle C(\mu),\mathcal{F}[\omega]\rangle$ is a continuous map on $\mathcal{D}_S^{\infty}(\mathfrak{g})$ for every compact $S\subset \mathfrak{g}$. Since $\mu$ is a finite measure on $\widehat{G}$, it is enough to show that for every $J(\Gamma)\in \widehat{G}$, the maps $\omega\mapsto \langle C(\Gamma),\mathcal{F}[\omega]\rangle$ are equicontinuous in $\Gamma$, in the sense that all of these expressions can be bounded in absolute value by a single seminorm on $\mathcal{D}_S^{\infty}(\mathfrak{g})$ whenever $\omega\in \mathcal{D}_S^{\infty}(\mathfrak{g})$.

As in the proof of Lemma \ref{singularbound}, we define
$$\widehat{G}_{\text{ext}}=\{J(\Gamma_1)|\ \exists\ \Gamma\ \text{s.t.}\ J(\Gamma)\in \widehat{G}\ \text{and}\ M_{\Gamma_1,\Gamma}\neq 0\}.$$
As in the proof of Lemma \ref{singularbound}, if $\Gamma=(H,\gamma,R_{i\mathbb{R}}^+)$ is a Langlands parameter, let $H(\Gamma):=H$ denote the associated Cartan subgroup. Moreover, if $H$ and $H_1$ are Cartan subgroups of $G$, define
$$(\widehat{G}_{\text{ext}})_{H,H_1}$$
$$=\{J(\Gamma_1)|\ H_1=H(\Gamma_1),\ \exists\ \Gamma\ \text{s.t.}\ M_{\Gamma_1,\Gamma}\neq \{0\}\ \&\ J(\Gamma)\in \widehat{G},\ H=H(\Gamma)\}.$$
Up to conjugacy, there are a finite number of Cartan subgroups $H\subset G$; therefore, we have divided the extended unitary dual $\widehat{G}_{\text{ext}}$ into a finite number of pieces $(\widehat{G}_{\text{ext}})_{H,H_1}$. We recall that we may write
$$C(\Gamma)=\sum_{\Gamma_1} M_{\Gamma_1,\Gamma}X(\Gamma_1).$$
Further, the coefficients $M_{\Gamma_1,\Gamma}$ are uniformly bounded over all Langlands parameters $\Gamma$ and $\Gamma_1$ for $G$. This follows from Vogan's work on the Jantzen filtration (see \cite{Vo84} for the original reference or Section 14 of \cite{ALTV} for an exposition) together with the Jantzen-Juckerman translation principle (see \cite{Zu77}, Chapter 7 of \cite{Vo81} for the original references or Section 16 of \cite{ALTV} for an exposition).

Combining all of this information, we come to the following conclusion. For every compact $S\subset \mathfrak{g}$, it is enough to give a bound on $\langle X(\Gamma_1),\mathcal{F}[\omega]\rangle$ for every $\omega\in \mathcal{D}_S^{\infty}(\mathfrak{g})$ in terms of seminorms of this space applied to $\omega$ and uniformly in $\Gamma_1$ with $J(\Gamma_1)\in (\widehat{G}_{\text{ext}})_{H,H_1}$. From now on, we fix two such Cartan subgroups $H,H_1\subset G$.

Recall that we have the decomposition $$X(\Gamma_1)=K\cdot \left(\mathcal{O}_{D(\Lambda_1)}+\nu_1+i(\mathfrak{g}/\mathfrak{p}_1)^*\right).$$
For each $k\in K$, define
$$X_k(\Gamma_1)=k\cdot \left(\mathcal{O}_{D(\Lambda_1)}+\nu_1+i(\mathfrak{g}/\mathfrak{p}_1)^*\right).$$
Since the measure on $K$ is invariant and therefore finite, we conclude that for every compact $S\subset \mathfrak{g}$, it is enough to give a bound on $\langle X_k(\Gamma_1),\mathcal{F}[\omega]\rangle$ for every $\omega\in \mathcal{D}_S^{\infty}(\mathfrak{g})$ in terms of seminorms of this space applied to $\omega$ and uniformly in $\Gamma_1$ with $J(\Gamma_1)\in (\widehat{G}_{\text{ext}})_{H,H_1}$ and uniformly in $k\in K$.

If $V$ is a real, finite dimensional vector space and $U\subset V$ is an open subset, let $\mathcal{S}(U)$ denote the Schwartz space on $U\subset V$. This is the space of all $\varphi\in C^{\infty}(U)$ for which $|D\varphi|$ is a bounded function on $U$ for every linear partial differential operator with polynomial coefficients $D$ on $V$. Let $H_1=H(\Gamma_1)$ be the Cartan subgroup corresponding to $\Gamma_1$, write $H_1=T_1A_1$ as a compact piece times a vector part, and let $M(\Gamma_1)=M_1$ be the standard reductive piece of the centralizer of $A_1$, $Z_G(A_1)=M_1A_1$. Let $\mathfrak{h}_1$ (resp. $\mathfrak{m}_1$) denote the Lie algebra of $H_1$ (resp. $M_1$). Recall Harish-Chandra's invariant integral map
$$\psi:\ \mathcal{S}(i\mathfrak{m}_1^*)\rightarrow \mathcal{S}((i\mathfrak{h}_1^*)')$$
by
$$\varphi\mapsto \psi_{\varphi}$$
where $$\psi_{\varphi}(\lambda)=\langle \varphi,\mathcal{O}^{M_1}_{\lambda} \rangle.$$

Harish-Chandra proved that his invariant integral map is continuous (see Theorem 3 of \cite{HC57}). Fix norms, both written $|\cdot|$, on the finite dimensional vector spaces $i\mathfrak{h}^*$ and $i\mathfrak{m}^*$. Harish-Chandra showed that there exists $l\in \mathbb{N}$ and a constant $b>0$ for which we have the estimate
$$\sup_{\xi\in i(\mathfrak{h}_1^*)'}|\psi_{\varphi}(\xi)|\leq b \sup_{\eta\in i\mathfrak{m}_1^*} (1+|\eta|)^l|\varphi(\eta)|.$$
This estimate follows from (the stronger estimate in) Lemma 7 on page 203 of \cite{HC57} (note that the constant $q$ in that Lemma can be chosen to be at least as large as one by the proof of Lemma 5 of \cite{HC57}). Further, by Lemma 22 on page 576 of \cite{HC64}, the function $\psi_{\varphi}$ extends to a smooth function on the closure of each Weyl chamber in $\mathfrak{h}^*$; the above bound extends to these values of Harish-Chandra's invariant integral by continuity. It follows from Supplement A and Supplement C on page 218 of \cite{Ro82} that for every map
$$\varphi\mapsto \langle \mathcal{O}_{D(\Lambda_1)}^{M_1},\varphi\rangle$$ 
with $\varphi\in \mathcal{S}(i\mathfrak{m}_1^*)$, there exists a closed Weyl chamber $i(\mathfrak{h}_1)^*_+\subset i\mathfrak{h}_1^*$ and (a possibly singular) $\lambda_0\in i(\mathfrak{h}_1)^*_+$ for which this map can be written as 
$$\varphi\mapsto \lim_{\substack{\lambda\rightarrow \lambda_0\\ \lambda\in i((\mathfrak{h}_1)^*_+)'}}\psi_{\varphi}(\lambda).$$ It follows that we have a uniform bound on $|\langle \mathcal{O}_{D(\Lambda_1)}^{M_1},\varphi\rangle|$ over all $\Lambda_1$.

Next, we define 
$$\widetilde{X}_k(\Gamma_1)=k\cdot \left(\mathcal{O}_{D(\Lambda_1)}^{M_1}+\operatorname{Im}\nu_1+i(\mathfrak{g}/\mathfrak{p}_1)^*\right)$$
and we note $X_k(\Gamma_1)=\widetilde{X}_k(\Gamma_1)+k\cdot \operatorname{Re}\nu_1$. As in the proof of Lemma \ref{singularbound}, we have decomposed $\nu_1=\operatorname{Re}\nu_1+\operatorname{Im}\nu_1$ with $\operatorname{Re}\nu_1\in \mathfrak{a}^*$ and $\operatorname{Im}\nu_1\in i\mathfrak{a}^*$. We observe that each $\widetilde{X}_k(\Gamma_1)$ is a tempered distribution on $i\mathfrak{g}^*$ since it is a linear transformation applied to a product of tempered distributions on $i\mathfrak{m}^*$ and $i(\mathfrak{g}/\mathfrak{p})^*$. Further, the fact that $k\in K$ is compact combined with our previous uniform bound on the tempered distributions $\mathcal{O}^{M_1}_{D(\Lambda_1)}$ yields the following uniform bound on $\widetilde{X}_k(\Gamma_1)$. Then there exists a natural number $l\in \mathbb{N}$ and a positive constant $b>0$ such that
\begin{equation}\label{eq:Harish-Chandra_bound}
|\langle \widetilde{X}_k(\Gamma_1),\varphi\rangle|\leq b\sup_{\xi\in i\mathfrak{g}^*} (1+|\xi|)^l |\varphi(\xi)|
\end{equation}
for all $k\in K$, all Langlands parameters $\Gamma_1$ with $J(\Gamma_1)\in (\widehat{G}_{\text{ext}})_{H,H_1}$, and all $\varphi\in \mathcal{S}(i\mathfrak{g}^*)$.

All that is left to do is to deal with the translation by $k\cdot \operatorname{Re}\nu_1$. If $S\subset \mathfrak{g}$ is a compact set, then by the Paley-Wiener Theorem (see for instance page 181 of \cite{Hor83}), every smooth density $\omega$ supported in $S$ satisfies a sequence of estimates
\begin{equation}\label{eq:Paley_Wiener}
\left| \mathcal{F}[\omega](\xi)\right|\leq \frac{A_ne^{B|\operatorname{Re}\xi|}}{(1+|\operatorname{Im}\xi|)^n}.
\end{equation}
If $\mathcal{D}_S^{\infty}(\mathfrak{g})$ is the space of smooth densities supported in $S\subset \mathfrak{g}$, then the constant $B>0$ depends on $S$ but not on $\omega\in \mathcal{D}_S^{\infty}(\mathfrak{g})$. Moreover, there exist seminorms $d_n$ on $\mathcal{D}_S^{\infty}(\mathfrak{g})$ such that $A_n(\omega)\leq d_n(\omega)$ for all $n\in \mathbb{N}$ and $\omega\in \mathcal{D}_S^{\infty}(\mathfrak{g})$ (see for instance page 181 of \cite{Hor83}).

By Lemma \ref{realpartlemma}, the set of all $k\cdot \nu_1$ with $k\in K$ and $J(\Gamma_1)\in (\widehat{G}_{\text{ext}})_{H,H_1}$ is contained in a bounded subset of $\mathfrak{h}_1^*$. In particular, $|k\cdot \nu_1|\leq c$ is bounded by a positive constant $c>0$ for all $\Gamma_1$ with $J(\Gamma_1)\in (\widehat{G}_{\text{ext}})_{H,H_1}$ and all $k\in K$. Thus, if $\omega\in \mathcal{D}_S^{\infty}(\mathfrak{g})$, then we deduce
$$|\langle X(\Gamma_1),\mathcal{F}[\omega]\rangle|\leq bA_le^{cB}$$
for every $\Gamma_1$ with $J(\Gamma_1)\in (\widehat{G}_{\text{ext}})_{H,H_1}$. Observe that the constants $b$, $c$, and $B$ do not depend on $\omega\in \mathcal{D}_S^{\infty}(\mathfrak{g})$. Further, the constant $A_l$ is bounded by a finite sum of seminorms on $\mathcal{D}_S^{\infty}(\mathfrak{g})$ as remarked above. It now follows that the linear functional
$$\omega\mapsto \langle C(\mu),\mathcal{F}[\omega]\rangle$$
defines a distribution on $\mathfrak{g}$ for every finite, positive measure $\mu$ on $\widehat{G}$. This proves part (1).


For part (2), we first show polynomial boundedness. Note 
$$(C(\mu)*\mathcal{F}[\omega])(\eta)=\langle C(\mu)_{\xi},\mathcal{F}[\omega](\eta-\xi)\rangle$$
if $\eta\in i\mathfrak{g}^*$. Utilizing (\ref{eq:Paley_Wiener}), we have
$$\left| \mathcal{F}[\omega](\eta-\xi)\right|\leq \frac{A_ne^{B|\operatorname{Re}(\eta-\xi)|}}{(1+|\operatorname{Im}(\eta-\xi)|)^n}\leq (1+|\operatorname{Im}\eta|)^n\frac{A_ne^{B|\operatorname{Re}\xi|}}{(1+|\operatorname{Im}\xi|)^n}.$$
The triangle inequality was used to obtain the inequality on the right. Then, following the argument above, we obtain
$$\left|(C(\mu)*\mathcal{F}[\omega])(\eta)\right|\leq c_{\omega} (1+|\operatorname{Im}\eta|)^l$$
for some constant $c_{\omega}>0$. 

Next, by the Paley-Weiner Theorem, one has bounds on the derivatives of $\mathcal{F}[\omega]$ analogous to the above bounds on $\mathcal{F}[\omega]$. These bounds allow us to differentiate under the integral sign. Smoothness of the convolution follows. 
\end{proof}
 
Suppose $V$ is a finite dimensional, real vector space, and let $v_1,\ldots,v_n$ be a basis for $V$. If $\alpha=(\alpha_1,\ldots,\alpha_n)$, denote 
$$D^{\alpha}=v_1^{\alpha_1}\cdots v_n^{\alpha_n},$$
which may be thought of as a differential operator on $V$. Suppose $0\in U_1\subset U$ are open, precompact sets in $V$ with $U_1$ compactly contained in $U$, and suppose $\{\varphi_{N,U_1,U}\}$ is a family of smooth functions depending on $N\in \mathbb{N}$ such that
\begin{enumerate} 
\item $\varphi_{N,U_1,U}(x)=1$ if $x\in U_1$ for all $N\in \mathbb{N}$.
\item $\varphi_{N,U_1,U}(x)=0$ if $x\notin U$ for all $N\in \mathbb{N}$.
\item For every multi-index $\alpha$, there exists a constant $C_{\alpha}>0$ such that
$$\sup_{x\in U}|D^{\alpha+\beta}\varphi_{N,U_1,U}(x)|\leq C_{\alpha}^{|\beta|+1}(N+1)^{|\beta|}$$
if $|\beta|\leq N$.
\end{enumerate}
Such families $\{\varphi_{N,U_1,U}\}$ always exist (see pages 25-26, 282 of \cite{Hor83}). 

\begin{lemma} \label{FphiN} Suppose $V$ is a finite dimensional real vector space, suppose $\varphi_{N,U_1,U}$ is a family of smooth, compactly supported functions on $V$ satisfying the above properties, and let $|\cdot|$ be a norm on $V_{\mathbb{C}}^*$. Write
$$\mathcal{F}[\varphi_{N,U_1,U}]=\widetilde{\mathcal{F}}[\varphi_{N,U_1,U}]d\xi$$
where $d\xi$ is a Lebesgue measure on $iV^*$ and $\widetilde{\mathcal{F}}[\varphi_{N,U_1,U}]$ is an analytic function on $iV^*$. Then there exist constants $B>0$ and $C>0$ such that
\begin{equation}\label{eq:Paley_Weiner_bound}
|\widetilde{\mathcal{F}}[\varphi_{N,U_1,U}](\xi)|\leq \frac{C^{N+1}(N+1)^Ne^{B|\operatorname{Re}\xi|}}{(1+|\operatorname{Im}\xi|)^{N}}
\end{equation}
for $\xi\in V_{\mathbb{C}}^*$.
\end{lemma} 

This estimate follows immediately from the proof of the Paley-Wiener Theorem (see the proof of Theorem 7.3.1 on page 181 of \cite{Hor83}).

\begin{lemma} [H\"{o}rmander] \label{Hormander} Let $V$ be a finite dimensional real vector space, and let $|\cdot|$ be a norm on $V_{\mathbb{C}}^*$, the complexification of the real dual space of $V$. Suppose we have a collection of pairs $\{(T_{\alpha},\xi_{\alpha})\}_{\alpha\in \mathcal{A}}$ with $T_{\alpha}$ a tempered distribution on $iV^*$ and $\xi_{\alpha}\in V^*$ for every $\alpha\in \mathcal{A}$. In addition, suppose they satisfy the following properties.  
\begin{enumerate}
\item There exists an upper bound $R$ on the collection of real numbers $|\xi_{\alpha}|$.
\item There exists a constant $b>0$ and a natural number $l\in \mathbb{N}$ such that 
$$\left|\langle T_{\alpha}, \varphi\rangle\right|\leq b\sup_{\xi\in iV^*} (1+|\xi|)^l |\varphi(\xi)|$$
for every Schwartz function $\varphi\in \mathcal{S}(iV^*)$ and for every $\alpha\in \mathcal{A}$.
\end{enumerate}
If $\xi\in V_{\mathbb{C}}^*$, write $\xi=\operatorname{Re}\xi+\operatorname{Im}\xi$ with $\operatorname{Re}\xi\in V^*$ and $\operatorname{Im}\xi\in iV^*$. Suppose $\psi\in C^{\infty}(V_{\mathbb{C}}^*)$ and suppose that for every $r>0$ and every $N\in \mathbb{N}$, there exists a constant $b_N(r)>0$ such that
\begin{equation}\label{eq:Hormander_bound1}
|\psi(\xi)|\leq b_N(r)(1+|\operatorname{Im}\xi|)^{-N}
\end{equation}
whenever $|\operatorname{Re}\xi|\leq r$. 
If $\Omega=\bigcup_{\alpha} \operatorname{supp}T_{\alpha}$ and $\eta\in iV^*\setminus \operatorname{AC}(\Omega)$, then there exists an open set $\eta\in W\subset iV^*$ and a constant $c>0$ such that for every natural number $N\in \mathbb{N}$, we have
\begin{equation}\label{eq:Hormander_bound2}
\left|\langle (T_{\alpha})_{\xi},\psi(t\eta'-(\xi-\xi_{\alpha}))\rangle\right|\leq c\cdot b_{N+l}(R)(1+t)^{-N}
\end{equation}
for all $\alpha\in \mathcal{A}$, all $\eta'\in W$ and all $t>0$.
\end{lemma}

\begin{proof}
This Lemma is a slight generalization of Lemma 8.1.7 of \cite{Hor83}. For the convenience of the reader, we give a self-contained argument. Since $\operatorname{AC}(\Omega)\subset iV^*$ is a closed cone and $\eta\notin \operatorname{AC}(\Omega)$, we may find open cones $\mathcal{C},\mathcal{C}_1\subset iV^*$ with 
\[\eta\in \mathcal{C}\subset \overline{\mathcal{C}}\setminus \{0\}\subset iV^*\setminus \operatorname{AC}(\Omega).\]
Following H\"{o}rmander, we claim there exist $C'>0$ and $\epsilon'>0$ such that 
\begin{equation}\label{eq:Hormander_cone}
|\eta'-\xi|\geq \epsilon' |\eta'|\ \text{if}\ \xi\in \Omega,\ \eta'\in \overline{\mathcal{C}},\ \text{and}\ |\eta'|> C'.
\end{equation}
If not, one could find sequences $\{\xi_j\}\subset \Omega$ and $\{\eta_j'\}\subset \overline{\mathcal{C}}$ with $|\eta_j'-\xi_j|<|\eta_j'|/j$ and $|\eta_j'|>j$. By compactness of the unit sphere, one deduces that the sequence $\{\xi_j/|\eta_j'|\}$ has a limit point in $\operatorname{AC}(\Omega)\cap (\overline{\mathcal{C}}\setminus \{0\})$. This is a contradiction since these two sets have empty intersection by definition.


Next, we claim that we may choose a sufficiently small open subset $\eta\in W\subset \mathcal{C}$ and constants $\epsilon>0$ and $C>0$ for which
\begin{equation}\label{eq:modified_cone}
|t\eta'-\xi|\geq t\epsilon \ \text{if}\ \eta'\in W,\ \xi\in \Omega,\ \text{and}\ t\geq C.
\end{equation}
Indeed, for $\delta>0$, define $W_{\delta}=\{\eta'\in \mathcal{C}\mid e^{-\delta}|\eta|<|\eta'|<e^{\delta}|\eta|\}$. If $\eta'\in W_{\delta}$ with $|t\eta'|\geq C'$, then by (\ref{eq:Hormander_cone}), 
\[|t\eta'-\xi|\geq \epsilon'|t\eta'|\geq t\epsilon'e^{-\delta}|\eta|\geq t\epsilon\]
where $\epsilon:=\epsilon'e^{-\delta}$. Further, $|t\eta'|\geq C'$ if $te^{\delta}|\eta|\geq C'$ if $t\geq C'e^{-\delta}|\eta|^{-1}$. Therefore, if $C:=C'e^{-\delta}|\eta|^{-1}$ and $\epsilon:=\epsilon'e^{-\delta}$, then (\ref{eq:modified_cone}) follows from (\ref{eq:Hormander_cone}) with $W:=W_{\delta}$ for some fixed $\delta>0$.

Now, fix a cutoff function $\chi\in C^{\infty}(iV^*)$ with $\chi(\xi)=1$ if $|\xi|\geq 1$, $\chi(\xi)=0$ if $|\xi|\leq \frac{1}{2}$, and $0\leq \chi(\xi)\leq 1$ for all $\xi\in iV^*$. Moreover, define $\chi_s(\xi)=\chi(\xi/s)$ for $s>0$. Next, choose $\psi$ satisfying (\ref{eq:Hormander_bound1}). Notice if $t\geq C$ and $s\leq t\epsilon$, then by (\ref{eq:modified_cone})
\[\left|\langle (T_{\alpha})_{\xi}, \psi(t\eta'-(\xi-\xi_{\alpha}))\rangle\right|=\left|\langle (T_{\alpha})_{\xi}, \chi_s(t\eta'-\xi)\psi(t\eta'-(\xi-\xi_{\alpha}))\rangle\right|.\]
Next, set $s=t\epsilon$ plug in $\varphi(\xi):=\chi(t\eta'-\xi)\psi(t\eta'-(\xi-\xi_{\alpha}))$ into the assumption (2) on $\{T_{\alpha}\}_{\alpha\in \mathcal{A}}$ and utilize (\ref{eq:Hormander_bound1}) to obtain
\begin{align*}&\left|\langle (T_{\alpha})_{\xi}, \chi_s(t\eta'-\xi)\psi(t\eta'-(\xi-\xi_{\alpha}))\rangle\right|\\
\leq &b \sup_{\xi\in iV^*} (1+|\xi|)^l|\chi_s(t\eta'-\xi)||\psi(t\eta'-(\xi-\xi_{\alpha}))|\\
\leq &b\cdot b_N(R)\sup_{\substack{\xi\in iV^*\\ |t\eta'-\xi|\geq \frac{t\epsilon}{2}}} (1+|\xi|)^l(1+|t\eta'-\xi|)^{-N}.\\
\leq &b\cdot b_N(R)\left(1+t\left(|\eta'|+\frac{\epsilon}{2}\right)\right)^l\left(1+\frac{t\epsilon}{2}\right)^N. \numberthis \label{eq:almost}
\end{align*} 
In the last step, we utilized the triangle inequality to check $|t\eta'-\xi|\geq \frac{t\epsilon}{2}$ implies $|\xi|\leq |t\eta'|+\frac{t\epsilon}{2}$. Next, if we rechoose $\epsilon<2$, then (\ref{eq:almost}) yields
\begin{align*}\left|\langle (T_{\alpha})_{\xi}, \psi(t\eta'-(\xi-\xi_{\alpha}))\rangle\right|&\leq b\cdot b_N(R)\cdot \max{\left(1,2e^{\delta}|\eta|,\epsilon\right)}^l\left(1+t\right)^l\left(1+t\right)^{-N}\\ 
&\leq c\cdot b_N(R)\cdot (1+t)^{-(N-l)}
\end{align*}
where $c>0$ is a constant independent of $N$. Replacing $N-l$ by $N$, we obtain
\[\left|\langle (T_{\alpha})_{\xi}, \psi(t\eta'-(\xi-\xi_{\alpha}))\rangle\right|\leq c\cdot b_{N+l}(R)\cdot (1+t)^{-N}.\]
This is the desired expression (\ref{eq:Hormander_bound2}).
\end{proof}





Next, we consider the singular spectrum of our integral $\theta(\mu)$.

\begin{lemma} \label{sslemma} If $\mu$ is a finite, positive measure on $\widehat{G}$, then the singular spectrum of the integral 
$$\theta(\mu)=\int_{J(\Gamma)\in \widehat{G}}\theta(\Gamma)d\mu_{\Gamma}$$
at zero is subject to the bound
$$\operatorname{SS}_0(\theta(\mu))\subset \operatorname{AC}\left(\bigcup_{J(\Gamma)\in \operatorname{supp}\mu} \operatorname{supp}C(\Gamma)\right).$$
\end{lemma}

\begin{proof}
To prove Lemma \ref{sslemma}, we recall the definition of the singular spectrum of a distribution (see Definition 2.3 of \cite{HHO}). Suppose $V$ is a finite dimensional real vector space, and 
let $0\in U_1\subset U$ be precompact open sets in $V$ with $U_1$ compactly contained in $U$. Fix a sequence $\varphi_{N,U_1,U}$ of smooth functions supported in $U$ and satisfying the properties given before Lemma \ref{FphiN}. Suppose $u$ is a distribution on the vector space $V$, and write 
$$\mathcal{F}[\varphi_{N,U_1,U}u]=\widetilde{\mathcal{F}}[\varphi_{N,U_1,U}u]d\xi$$
where $d\xi$ is a Lebesgue measure on $i\mathfrak{g}^*$ and $\widetilde{\mathcal{F}}[\varphi_{N,U_1,U}u] \in C^{\infty}(i\mathfrak{g}^*)$ is a smooth function on $i\mathfrak{g}^*$. Then we may compute the singular spectrum of $u$ at $0$ by recalling $(0,\xi)$ with $\xi\neq 0$ is not in the singular spectrum if and only if there exists an open set $\xi\in W\subset iV^*$ and a constant $C>0$ such that
$$|\widetilde{\mathcal{F}}[\varphi_{N,U_1,U}u](t\eta)|\leq C^{N+1}(N+1)^Nt^{-N}$$ for every $\eta\in W$.

The distribution we wish to study is $\theta(\mu)$ where $\mu$ is a finite, positive measure on the unitary dual, $\widehat{G}$. First, we check 
$$\mathcal{F}[\theta(\mu)\varphi_{N,U_1,U}]=C(\mu)*\mathcal{F}[\varphi_{N,U_1,U}].$$ 
We will do this by manipulating the right hand side to obtain the left hand side.  
By part (2) of Lemma \ref{chint}, the left hand side converges absolutely to a smooth, polynomially bounded density on $i\mathfrak{g}^*$. In particular, it is a distribution on $i\mathfrak{g}^*$.

Now, the Fourier transform of the tempered distribution $C(\mu)*\mathcal{F}[\varphi_{N,U_1,U}]$ is by definition the tempered distribution 
$$\omega\mapsto \langle C(\mu)*\mathcal{F}[\varphi_{N,U_1,U}], \mathcal{F}[\omega]\rangle$$
where $\omega$ is in the space of Schwartz densities on $\mathfrak{g}$. Since we have already checked absolute convergence of this integral, we can formally transfer the convolution to obtain
$$\langle C(\mu), \mathcal{F}[\varphi_{N,U_1,U}]*\mathcal{F}[\omega]\rangle.$$
Now, since $\mathcal{F}[\varphi_{N,U_1,U}]$ is a Schwartz density on $i\mathfrak{g}^*$ and $\mathcal{F}[\omega]$ is a Schwartz function on $i\mathfrak{g}^*$, we note
$$\mathcal{F}[\varphi_{N,U_1,U}]*\mathcal{F}[\omega]=\mathcal{F}[\varphi_{N,U_1,U}\cdot \omega]$$
on $i\mathfrak{g}^*$. However, this is also true on $\mathfrak{g}_{\mathbb{C}}^*$ since both sides have holomorphic extensions to $\mathfrak{g}_{\mathbb{C}}^*$, which are determined by their restrictions to $i\mathfrak{g}^*$. 
Now, we finish by observing
$$\langle C(\mu),\mathcal{F}[\varphi_{N,U_1,U}\omega]\rangle=\langle \theta(\mu),\varphi_{N,U_1,U}\omega\rangle=\langle \theta(\mu)\varphi_{N,U_1,U},\omega\rangle.$$
As remarked earlier, this follows from work of Duflo and Rossmann.

Now that we have checked that $C(\mu)*\mathcal{F}[\varphi_{N,U_1,U}]$ is the Fourier transform of $\theta(\mu)\varphi_{N,U_1,U}$, we can estimate the singular spectrum of $\theta(\mu)$ by estimating the decay of the former distribution.

More precisely, if $\eta\in i\mathfrak{g}^*\setminus \operatorname{AC}\left(\bigcup_{J(\Gamma)\in \operatorname{supp}\mu} \operatorname{supp}C(\Gamma)\right)$, we must find an open neighborhood $\eta\in W\subset i\mathfrak{g}^*$ and a constant $C>0$ such that
$$\left|[C(\mu)*\widetilde{\mathcal{F}}[\varphi_{N,U_1,U}]](t\eta)\right|=\left|\int_{J(\Gamma)\in \widehat{G}} \int_{C(\Gamma)} \widetilde{\mathcal{F}}[\varphi_{N,U_1,U}](t\eta'-\xi)d\xi d\mu_{\Gamma}\right|$$
$$\leq C^{N+1}(N+1)^Nt^{-N}$$
for all $\eta'\in W$ and $N\in \mathbb{N}$. Here as before, we are writing 
$$\mathcal{F}[\varphi_{N,U_1,U}]=\widetilde{\mathcal{F}}[\varphi_{N,U_1,U}]d\xi$$
where $d\xi$ is a Lebesgue measure on $i\mathfrak{g}^*$ and $\widetilde{\mathcal{F}}[\varphi_{N,U_1,U}]$ is a smooth function on $i\mathfrak{g}^*$, which extends to a holomorphic function on $\mathfrak{g}_{\mathbb{C}}^*$.
Since the measure $\mu$ is finite, it is enough to show that there exists a neighborhood $\eta\in W\subset i\mathfrak{g}^*$ and a constant $C>0$ such that $$\left|\langle C(\Gamma)_{\xi}, \widetilde{\mathcal{F}}[\varphi_{N,U_1,U}](t\eta'-\xi)\rangle\right|\leq  C^{N+1}(N+1)^Nt^{-N}$$ 
for every $\eta'\in W$, $J(\Gamma)\in \operatorname{supp}\mu$, and $N\in \mathbb{N}$. Next, suppose $M_{\Gamma_1,\Gamma}\neq 0$ for some $\Gamma_1$ for which $J(\Gamma)\in \operatorname{supp}\mu$, and recall
$$X(\Gamma_1)=K\cdot (\mathcal{O}_{D(\Lambda_1)}^{M_1}+\nu_1+i(\mathfrak{g}/\mathfrak{p}_1)^*).$$
As in Section 2, for every $k\in K$, define
$$X_k(\Gamma_1)=k\cdot (\mathcal{O}_{D(\Lambda_1)}^{M_1}+\nu_1+i(\mathfrak{g}/\mathfrak{p}_1)^*).$$
Then it is enough to show that there exists a neighborhood $\eta\in W\subset i\mathfrak{g}^*$ and a constant $C>0$ such that 
\begin{equation}\label{eq:final_bound}
\left|\langle X_k(\Gamma_1)_{\xi}, \widetilde{\mathcal{F}}[\varphi_{N,U_1,U}](t\eta'-\xi)\rangle\right|\leq  C^{N+1}(N+1)^Nt^{-N}
\end{equation}
for every $\eta'\in W$, $N\in \mathbb{N}$, $k\in K$, and $\Gamma_1$ for which $M_{\Gamma_1,\Gamma}\neq 0$ for some $\Gamma$ with $J(\Gamma)\in \operatorname{supp}\mu$. 

Justifying the last remark requires the fact that the coefficients $M_{\Gamma_1,\Gamma}$ are bounded uniformly for all Langlands parameters $\Gamma$ and $\Gamma_1$ and for a fixed real, reductive algebraic group $G$. This follows from Vogan's work on the Jantzen filtration (see \cite{Vo84} for the original reference or Section 14 of \cite{ALTV} for an exposition) together with the Jantzen-Juckerman translation principle (see \cite{Zu77}, Chapter 7 of \cite{Vo81} for the original references or Section 16 of \cite{ALTV} for an exposition).

Now, to establish (\ref{eq:final_bound}), we must apply Lemma \ref{FphiN} and Lemma \ref{Hormander}. In Lemma \ref{Hormander}, we let $\mathcal{A}$ be the collection of pairs $(\Gamma_1,k)$ with $M_{\Gamma_1,\Gamma}\neq 0$ for some $J(\Gamma)\in \operatorname{supp}\mu$ and $k\in K$. If $\alpha=(\Gamma_1,k)$, we let $T_{\alpha}=T_{(\Gamma_1,k)}$ be the tempered distribution
$$\widetilde{X}_k(\Gamma_1):=k\cdot (\mathcal{O}_{D(\Lambda_1)}^{M_1}+\operatorname{Im}\nu_1+i(\mathfrak{g}/\mathfrak{p}_1)^*),$$
and we let $\xi_{\alpha}=\xi_{(\Gamma_1,k)}:=k\cdot \operatorname{Re}\nu_1$. We must check that $\{(T_{\alpha},\xi_{\alpha})\}_{\alpha\in \mathcal{A}}$ satisfy properties (1) and (2) of Lemma \ref{Hormander}. The first property follows from Lemma \ref{realpartlemma}. The second property follows from (\ref{eq:Harish-Chandra_bound}).

Now, we put $\psi=\widetilde{\mathcal{F}}[\varphi_{N,U_1,U}]$ in Lemma \ref{Hormander}, and we compare the constants in equations (\ref{eq:Paley_Weiner_bound}) and (\ref{eq:Hormander_bound1}) to determine
\[b_N(r)=C^{N+1}(N+1)^Ne^{Br}.\]
Plugging this expression for $b_N(r)$ into (\ref{eq:Hormander_bound2}), we obtain
\begin{align*}
\left|\langle X_k(\Gamma_1)_{\xi}, \widetilde{\mathcal{F}}[\varphi_{N,U_1,U}](t\eta'-\xi)\rangle\right|&=\left|\langle \widetilde{X}_k(\Gamma_1)_{\xi}, \widetilde{\mathcal{F}}[\varphi_{N,U_1,U}](t\eta'-(\xi-\xi_{\alpha}))\rangle\right|\\
&\leq  c\cdot b_{N+l}(R)(1+t)^{-N} \\&= c\cdot e^{BR}\cdot C^{N+l+1}(N+l+1)^{l+N}(1+t)^{-N}\\
&\leq (C')^{N+l+1}(N+l+1)^{l+N}t^{-N} \numberthis \label{eq:nearly}
\end{align*}
for a constant $C'>0$ that is independent of $N$. If $N>l$, then
\begin{align*}
(N+l+1)^{N+l}&\leq (N+1)^N \left(\frac{l+N+1}{N+1}\right)^N (l+N+1)^{l}\\
&\leq (N+1)^N 2^N (2N)^{l}\\
&\leq (N+1)^N 2^N (2N+3)^{l+3}\\
&\leq (N+1)^N 2^N (l+3)^{2N+3}. \numberthis \label{eq:nearly2}
\end{align*}
In the last step, we used that $s^r>r^s$ if $r>s\geq 3$ are natural numbers; one checks this by observing $\log r/r$ is a decreasing function for $r>e$. Plugging (\ref{eq:nearly2}) into (\ref{eq:nearly}), we see there exists a constant $C''>0$ independent of $N$ such that
$$\left|\langle X_k(\Gamma_1)_{\xi}, \widetilde{\mathcal{F}}[\varphi_{N,U_1,U}](t\eta'-\xi)\rangle\right|\leq  (C'')^{N+1}(N+1)^{N}t^{-N}$$
if $N>l$. The Lemma follows.
\end{proof}

The proof of Theorem \ref{maintheorem} is nearly complete. Thus far, we have shown that if 
$$\pi\simeq \int_{J(\Gamma)\in \widehat{G}_s} J(\Gamma)^{\oplus m(J(\Gamma),\pi)} d\mu_{\Gamma}$$
is a direct integral of singular representations and $\mu'$ is a finite, positive measure that is equivalent to $\mu$, then
$$\operatorname{SS}_0(\theta(\mu'))\subset i(\mathfrak{g}_s^*)'.$$
The final Lemma we need is essentially already in \cite{HHO}. 

\begin{lemma} \label{wfch} Suppose $G$ is a Lie group, suppose $(\pi,V)$ is a unitary representation of $G$, and write
$$\pi\simeq \int_{\sigma\in \widehat{G}} \sigma^{\oplus m(\pi,\sigma)}d\mu_{\sigma}$$
as a direct integral of irreducible representations. Suppose that for every $\sigma\in \operatorname{supp}\pi$, the character $\Theta_{\sigma}$ exists as a distribution on $G$, and suppose that the integral
$$\Theta_{\mu'}=\int_{\sigma\in \widehat{G}} \Theta_{\sigma}d\mu'_{\sigma}$$
exists as a distribution on a neighborhood of the identity in $G$ for all finite, positive measures $\mu'$ on $\widehat{G}$ that are equivalent to $\mu$. If $u,v\in V$ are vectors, then 
$$\operatorname{SS}_e(\pi(g)u,v)\subset \bigcup_{\substack{\mu'\sim \mu\\ \mu'\ \text{finite,\ positive}}} \operatorname{SS}_e(\Theta_{\mu'})$$
where $\mu'\sim \mu$ means that $\mu'$ and $\mu$ are absolutely continuous with respect to each other.
\end{lemma}

Though not stated explicitly, this Lemma is proved in Section 6 of \cite{HHO}. There, it is assumed that $G$ is reductive and $\operatorname{supp}\pi$ is contained in the set of irreducible, tempered representations. However, these assumptions are not utilized in the proof. Because Lemma \ref{wfch} is not stated explicitly in \cite{HHO}, we include a brief sketch for the reader below.

First, we remark that in our setting, $\operatorname{SS}_0(\theta(\mu'))=\operatorname{SS}_e(\Theta_{\mu'})$. Hence, Theorem \ref{maintheorem} now follows from Lemma \ref{singularbound}, Lemma \ref{sslemma}, and Lemma \ref{wfch}. 
\bigskip

Now, we briefly recall the outline of the proof of Lemma \ref{wfch}, which is essentially already in Section 7 of \cite{HHO}. We choose $\varphi_{N,U_1,U}$ as in Section 7 of \cite{HHO}. As in the proof of Proposition 7.1 of \cite{HHO}, we write $u=(u_{\sigma})$ and $v=(v_{\sigma})$ using our direct integral decomposition. And the computation that follows in \cite{HHO} yields
$$\left|\int_G \varphi_{N,U_1,U}(g)(\pi(g)u,v)e^{it\eta(\log(g))}dg\right|$$
$$\leq \left(\int_{\sigma\in \operatorname{supp}\pi} \left|\sigma(\varphi_{N,U_1,U}e^{it\eta\log(\cdot)})\right|_{\operatorname{HS}}^2 |u_{\sigma}|^2 d\mu(\sigma)\right)^{1/2}\left(\int_{\sigma\in \operatorname{supp}\pi} |v_{\sigma}|^2 d\mu(\sigma)\right)^{1/2}.$$
Here $|\cdot|_{\operatorname{HS}}$ denotes the Hilbert Schmidt norm of the operator. Next, using a calculation of Howe \cite{How}, one writes the Hilbert Schmidt norm of the operator as an integral involving the character (see the proof of Proposition 7.1 in \cite{HHO}). Using a technical Lemma from microlocal analysis, one is able to show that this integral decays sufficiently rapidly to imply that 
$$\eta\notin \operatorname{SS}_e(\pi(g)u,v)$$ whenever 
$$\eta\notin \bigcup_{\substack{\mu'\sim \mu\\ \mu'\ \text{finite,\ positive}}} \operatorname{SS}_e(\Theta_{\mu'})$$
(see the last page of Section 7 of \cite{HHO}). Lemma \ref{wfch} follows.

\section{Applications to Induction Problems}

We begin by explaining how Corollary \ref{inducedcor} and Corollary \ref{L^2cor} follow from Theorem \ref{maintheorem} together with Theorem 1.1 of \cite{HHO} and Theorem 1.2 of \cite{HHO}. Suppose $G$ is a real, reductive algebraic group, suppose $H\subset G$ is a closed subgroup, and suppose $(\tau,W)$ is an irreducible, unitary representation of $H$. Then we may write 
$$\operatorname{Ind}_H^G\tau\simeq\int_{J(\Gamma)\in \widehat{G}} J(\Gamma)^{\oplus m(J(\Gamma),\operatorname{Ind}_H^G\tau)} d\mu_{\Gamma}$$
with $\mu$ a positive measure on $\widehat{G}$. If $\widehat{G}_{\text{temp}}^{\text{\ }\prime}\subset \widehat{G}$ is the open subset of tempered representations with regular infinitesimal character, then we may decompose
$$\mu=\mu|_{\widehat{G}_{\text{temp}}^{\text{\ }\prime}}+\mu|_{\widehat{G}-\widehat{G}_{\text{temp}}^{\text{\ }\prime}}.$$
This gives rise to a decomposition
$$\operatorname{Ind}_H^G\tau\simeq \int_{J(\Gamma)\in \widehat{G}_{\text{temp}}^{\text{\ }\prime}} J(\Gamma)^{\oplus m(J(\Gamma),\operatorname{Ind}_H^G\tau)} d\mu|_{\widehat{G}_{\text{temp}}^{\text{\ }\prime}}$$
$$\bigoplus \int_{J(\Gamma)\in \widehat{G}-\widehat{G}_{\text{temp}}^{\text{\ }\prime}} J(\Gamma)^{\oplus m(J(\Gamma),\operatorname{Ind}_H^G\tau)} d\mu|_{\widehat{G}-\widehat{G}_{\text{temp}}^{\text{\ }\prime}}.$$
Every matrix coefficient of $\operatorname{Ind}_H^G\tau$ can be decomposed into the sum of a matrix coefficient of the first representation plus a matrix coefficient of the second representation. As in Proposition 1.3 of \cite{How}, we deduce
$$\operatorname{WF}(\operatorname{Ind}_H^G\tau)$$
is the union of the 
$$\operatorname{WF}\left(\int_{J(\Gamma)\in \widehat{G}_{\text{temp}}^{\text{\ }\prime}} J(\Gamma)^{\oplus m(J(\Gamma),\operatorname{Ind}_H^G\tau)} d\mu|_{\widehat{G}_{\text{temp}}^{\text{\ }\prime}}\right)$$
and 
$$\operatorname{WF}\left(\int_{J(\Gamma)\in \widehat{G}-\widehat{G}_{\text{temp}}^{\text{\ }\prime}} J(\Gamma)^{\oplus m(J(\Gamma),\operatorname{Ind}_H^G\tau)} d\mu|_{\widehat{G}-\widehat{G}_{\text{temp}}^{\text{\ }\prime}}\right).$$
The same goes for the singular spectrum of $\operatorname{Ind}_H^G\tau$. Now, we know from Theorem 1.1 of \cite{HHO} that
$$\operatorname{WF}(\operatorname{Ind}_H^G\tau)\supset \overline{\operatorname{Ad}^*(G)\cdot q^{-1}(\operatorname{WF}(\tau))}$$
where $\mathfrak{g}$ (resp. $\mathfrak{h}$) denotes the Lie algebra of $G$ (resp. $H$) and $q\colon i\mathfrak{g}^*\rightarrow i\mathfrak{h}^*$ is the natural projection. Further, the analogous statement for the singular spectrum is also part of Theorem 1.1 of \cite{HHO}. Now, by Theorem 1.1 of this paper, we deduce
$$\operatorname{WF}\left(\int_{J(\Gamma)\in \widehat{G}-\widehat{G}_{\text{temp}}^{\text{\ }\prime}} J(\Gamma)^{\oplus m(J(\Gamma),\operatorname{Ind}_H^G\tau)} d\mu|_{\widehat{G}-\widehat{G}_{\text{temp}}^{\text{\ }\prime}}\right)\subset i\mathfrak{g}_s^*.$$
Again, the analogous statement holds for the singular spectrum as well. Thus, we deduce
$$\operatorname{WF}\left(\int_{J(\Gamma)\in \widehat{G}_{\text{temp}}^{\text{\ }\prime}} J(\Gamma)^{\oplus m(J(\Gamma),\operatorname{Ind}_H^G\tau)} d\mu|_{\widehat{G}_{\text{temp}}^{\text{\ }\prime}}\right)\supset \overline{\operatorname{Ad}^*(G)\cdot q^{-1}(\operatorname{WF}(\tau))}\cap i(\mathfrak{g}^*)'.$$
Again, this result still holds if we replace $\operatorname{WF}$ with $\operatorname{SS}$ everywhere. Finally, we apply Theorem 1.2 of \cite{HHO} to deduce
$$\operatorname{AC}\left(\bigcup_{\substack{\sigma\in \operatorname{supp}\operatorname{Ind}_H^G\tau\\ \sigma\in \widehat{G}_{\text{temp}}^{\text{\ }\prime}}}\mathcal{O}_{\sigma}\right)\supset \overline{\operatorname{Ad}^*(G)\cdot q^{-1}(\operatorname{WF}(\tau))}\cap i(\mathfrak{g}^*)'.$$
Since Theorem 1.2 of \cite{HHO} is also stated for the singular spectrum, we deduce the identical statement with $\operatorname{WF}$ replaced by $\operatorname{SS}$, which is the statement of Corollary \ref{inducedcor}.

Next, we check Corollary 1.3. To do this, we note that when $\tau$ is the trivial representation of $H$, $\operatorname{WF}(\tau)=\operatorname{SS}(\tau)=\{0\}$ since all of the matrix coefficients of $\tau$ are necessarily analytic. Assuming $X=G/H$ has a non-zero $G$ invariant density, we deduce
$$\operatorname{AC}\left(\bigcup_{\substack{\sigma\in \operatorname{supp}L^2(X)\\ \sigma\in \widehat{G}_{\text{temp}}^{\text{\ }\prime}}}\mathcal{O}_{\sigma}\right)\supset \overline{\operatorname{Ad}^*(G)\cdot i(\mathfrak{g}/\mathfrak{h})^*}\cap i(\mathfrak{g}^*)'$$
since $i(\mathfrak{g}/\mathfrak{h})^*$ is the inverse image of zero under the map $q:i\mathfrak{g}^*\rightarrow i\mathfrak{h}^*$. Next, if $x\in X$, let $G_x$ denote the stabilizer of $x$ in $G$, and let $\mathfrak{g}_x$ denote the Lie algebra of $G_x$. For each $x\in X$, we consider the map $g\mapsto g\cdot x$. We note that the kernel of the associated map on tangent spaces
$$\mathfrak{g}=T_eG\rightarrow T_xX$$
is $\mathfrak{g}_x$. Thus, the image of the pullback map $iT_x^*X\rightarrow iT_e^*G=i\mathfrak{g}^*$ is the set of imaginary valued linear functionals on $\mathfrak{g}$ that vanish on $i\mathfrak{g}_x$. Thus, this injection identifies $iT_x^*X$ with $i(\mathfrak{g}/\mathfrak{g}_x)^*\subset i\mathfrak{g}^*$. Further, $\mathfrak{g}_{\overline{e}}=\mathfrak{h}$, and 
$$\operatorname{Ad}^*(g)\cdot i(\mathfrak{g}/\mathfrak{g}_{\overline{e}})^*=i(\mathfrak{g}/\mathfrak{g}_{\overline{g}})^*$$ 
where $\overline{g}$ denotes the image of $g$ under the map $G\rightarrow X$ by $g\cdot e$. Then we have
$$\bigcup_{x\in X} iT^*_xX=\bigcup_{g\in G} [\operatorname{Ad}^*(g) \cdot i(\mathfrak{g}/\mathfrak{h})^*]$$
and we immediately deduce Corollary \ref{L^2cor}.
\bigskip

Now, in order to use Corollary \ref{inducedcor} and Corollary \ref{L^2cor} to compute concrete examples, it is useful to write things down concretely on the spaces of purely imaginary valued linear functionals on Cartan subalgebras of $\mathfrak{g}$. Suppose $H\subset G$ is a Cartan subgroup with Lie algebra $\mathfrak{h}$, let $W=N_G(H)/H$ denote the real Weyl group of $H$, and let $i(\mathfrak{h}^*)'=i\mathfrak{h}^*\cap i(\mathfrak{g}^*)'$ be the set of regular, semisimple elements in $i\mathfrak{h}^*$. For each Langlands parameter $\Gamma=(H,\gamma,R_{i\mathbb{R}}^+)$ with $J(\Gamma)\in \widehat{G}_{\text{temp}}^{\text{\ }\prime}$, an irreducible, tempered representation with regular infinitesimal character, and $H(\Gamma)=H$, we associate to $J(\Gamma)$ the $W$ orbit $W\cdot d\gamma\subset (i\mathfrak{h}^*)'$ through $d\gamma$. If $\pi$ is any unitary representation of $G$, we define
$$i(\mathfrak{h}^*)'-\operatorname{supp}\pi=\bigcup_{\substack{J(\Gamma) \in \operatorname{supp}\pi \\J(\Gamma)\in \widehat{G}_{\text{temp}}^{\text{\ }\prime}\\ H(\Gamma)=H}}(W\cdot d\gamma)$$
by taking the union of all such $W$ orbits over all irreducible, tempered representations with regular infinitesimal character $J(\Gamma)$ occurring in the decomposition of $\pi$ into irreducibles and which correspond to the Cartan $H$.

Now, we have an analogue of Corollary \ref{L^2cor} on $i$ times the dual of every Cartan.

\begin{corollary} \label{cartanL^2cor} Suppose $G$ is a real, reductive algebraic group with Lie algebra $\mathfrak{g}$, and suppose $X$ is a homogeneous space for $G$ with a non-zero invariant density. Suppose $\mathfrak{h}\subset \mathfrak{g}$ is a Cartan subalgebra, suppose $\xi\in i(\mathfrak{h}^*)'$, and suppose there exists $x\in X$ such that
$$\xi|_{\mathfrak{g}_x}=0.$$
Then $$\xi\in \operatorname{AC}\left(i(\mathfrak{h}^*)'-\operatorname{supp}L^2(X)\right).$$
\end{corollary}

\begin{proof} We describe how to deduce this Corollary from Corollary \ref{L^2cor}. Suppose the hypotheses of Corollary \ref{cartanL^2cor} hold, and suppose $\xi\in \mathcal{C}\subset i\mathfrak{h}^*$ is an open cone in $i\mathfrak{h}^*$ containing $\xi$. We must show that $i(\mathfrak{h}^*)'-\operatorname{supp}L^2(X)$ intersects $\mathcal{C}$ in an unbounded set. Without loss of generality, we may assume $\mathcal{C}\subset i(\mathfrak{h}^*)'$ is small enough to be contained in the set of regular, semisimple elements. Let $e\in U\subset G$ be a precompact open neighorhood of the identity in $G$, and let $\mathcal{C}'=\operatorname{Ad}^*(U)\cdot \mathcal{C}$. We observe that $\xi\in \mathcal{C}'\subset i\mathfrak{g}^*$ is an open cone containing $\xi$ in $i\mathfrak{g}^*$.

Next, we note $\xi|_{\mathfrak{g}_x}=0$ implies that $\xi$ is on the right hand side of Corollary \ref{L^2cor}; hence, it must also be on the left hand side of Corollary \ref{L^2cor}. In particular, 
$$\xi\in \operatorname{AC}\left(\bigcup_{\substack{\sigma\in \operatorname{supp}L^2(X)\\ \sigma\in \widehat{G}_{\text{temp}}^{\text{\ }\prime}}}\mathcal{O}_{\sigma}\right).$$
Therefore,
$$\mathcal{C}'_X:=\mathcal{C}'\cap \left(\bigcup_{\substack{\sigma\in \operatorname{supp}L^2(X)\\ \sigma\in \widehat{G}_{\text{temp}}^{\text{\ }\prime}}}\mathcal{O}_{\sigma}\right)$$ 
is an unbounded set. However, the map $U\times \mathcal{C}\rightarrow i\mathfrak{g}^*$ is continuous and $U$ is precompact. Hence, 
$$\{(u,\eta)|\ u\in U,\ \eta\in \mathcal{C},\ \text{and}\ \operatorname{Ad}^*(u)\cdot \eta\in \mathcal{C}'_X\}$$
must contain an unbounded subset of $\eta\in \mathcal{C}$. However, since we are intersecting $\mathcal{C}'$ with a $G$ invariant set, the above set is independent of $u\in U$. In particular, $(e,\eta)$ occurs in the set for an unbounded subset of $\eta\in \mathcal{C}$ and 
$$i(\mathfrak{h}^*)'-\operatorname{supp}L^2(X)\cap \mathcal{C}$$
is unbounded. The Lemma follows.
\end{proof}

Next, we give an application of this result. We begin with the case where $\tau$ is an involution of $G$, $G^{\tau}\subset G$ is the symmetric subgroup of fixed points of $\tau$, and $H\subset G^{\tau}$ is a closed, unimodular subgroup. The decomposition of $L^2(G/G^{\tau})$ is already quite well understood. The representations of $G$ occurring discretely in $L^2(G/G^{\tau})$ were classified by Flensted-Jensen and Matsuki-Oshima \cite{FJ80}, \cite{MO84}; see also the exposition of Vogan \cite{Vo88}, which describes these representations in terms that are useful for our purposes. Proofs of the full Plancherel formula were given by Delorme \cite{De98} and van den Ban-Schlichtkrull \cite{BS05a}, \cite{BS05b}. 

Now, suppose $H\subset G^{\tau}$ is a closed, unimodular subgroup. One can ask about the decomposition of $L^2(G/H)$ into irreducibles; in particular, one can ask about the irreducible representations of $G$ occurring discretely in $L^2(G/H)$. Some results in this vein were given by Kobayashi \cite{Ko98}. We give a different yet analogous result.

\begin{corollary} \label{symmetriccor} Suppose $G$ is a real, reductive algebraic group, suppose $\tau$ is an involution of $G$ with fixed point set $G^{\tau}$, and suppose $H\subset G^{\tau}$ is a closed, unimodular subgroup of $G^{\tau}$. Assume that $G$ contains a compact Cartan subgroup $T$ with Lie algebra $\mathfrak{t}$. Then 
$$i(\mathfrak{t}^*)'\cap \operatorname{AC}\left(i(\mathfrak{t}^*)'-\operatorname{supp}L^2(G/G^{\tau})\right) \subset \operatorname{AC}\left(i(\mathfrak{t}^*)'-\operatorname{supp}L^2(G/H)\right).$$
\end{corollary}

This Corollary says that $L^2(G/H)$ has ``asymptotically more'' Harish-Chandra discrete series representations than $L^2(G/G^{\tau})$ when $H\subset G^{\tau}$ is a closed, unimodular subgroup. It follows directly from Lemma \ref{cartanL^2cor} together with the fact that the converse of Lemma \ref{cartanL^2cor} is also true for symmetric spaces in the case of a compact Cartan subalgebra $\mathfrak{t}\subset \mathfrak{g}$ (this can be read off of the description of the discrete spectrum in these symmetric spaces given in \cite{Vo88}). 

Let us look at the example where $G=\operatorname{Sp}(2n,\mathbb{R})$ and $G^{\tau}=\operatorname{GL}(n,\mathbb{R})$. If one chooses a compact Cartan subgroup $T\subset \operatorname{Sp}(2n,\mathbb{R})$, one can see either from Corollary \ref{cartanL^2cor} or work on the Plancherel formula for symmetric spaces (see for instance \cite{Vo88}) that
$$\operatorname{AC}\left(i(\mathfrak{t}^*)'-\operatorname{supp}L^2(\operatorname{Sp}(2n,\mathbb{R})/\operatorname{GL}(n,\mathbb{R}))\right)=i\mathfrak{t}^*.$$
Let $W=N_G(T)/T$ denote the Weyl group of $T$ in $G$. Then we deduce from Corollary \ref{symmetriccor} that whenever $H\subset \operatorname{GL}(n,\mathbb{R})$ is a closed unimodular subgroup and $\mathcal{C}\subset i\mathfrak{t}^*$ is an open cone in $i\mathfrak{t}^*$, then there are infinitely many Harish-Chandra discrete series $J(\Gamma)$ of $\operatorname{Sp}(2n,\mathbb{R})$ for which $(W\cdot d\gamma)\cap \mathcal{C}\neq \{0\}$ and
$$\operatorname{Hom}_{\operatorname{Sp}(2n,\mathbb{R})}(J(\Gamma),L^2(\operatorname{Sp}(2n,\mathbb{R})/H))\neq \{0\}.$$
As remarked in the introduction, one can take 
$$H=\prod_{i=1}^k \operatorname{GL}(p_i,\mathbb{R})\times \prod_{j=1}^l \operatorname{GL}(q_j,\mathbb{Z})$$
whenever $\sum p_i+\sum q_j\leq n$ and one can take $\mathcal{C}$ to be a Weyl chamber in $i\mathfrak{t}^*$.
\bigskip

We conclude this section by remarking that the recent work of Benoist and Kobayashi \cite{BK15} provides a wealth of additional examples where the right hand side of Corollary \ref{L^2cor} is nonempty.




\section{Applications to Restriction Problems}

Next, we consider the case of restriction to a real, reductive algebraic group. Let $G$ be a Lie group with Lie algebra $\mathfrak{g}$, let $H\subset G$ be a closed, real, reductive algebraic subgroup with Lie algebra $\mathfrak{h}$,  and let $\pi$ be a unitary representation of $G$.
We may write 
$$\pi|_H\simeq \int_{J(\Gamma)\in \widehat{H}} J(\Gamma)^{\oplus m(J(\Gamma),\pi|_H)} d\mu_{\Gamma}$$
with $\mu$ a positive measure on $\widehat{H}$. If $\widehat{H}_{\text{temp}}^{\text{\ }\prime}\subset \widehat{H}$ is the open subset of tempered representations with regular infinitesimal character, then we may decompose
$$\mu=\mu|_{\widehat{H}_{\text{temp}}^{\text{\ }\prime}}+\mu|_{\widehat{H}-\widehat{H}_{\text{temp}}^{\text{\ }\prime}}.$$
This gives rise to a decomposition
$$\pi|_H\simeq \int_{J(\Gamma)\in \widehat{H}_{\text{temp}}^{\text{\ }\prime}} J(\Gamma)^{\oplus m(J(\Gamma),\pi|_H)} d\mu|_{\widehat{H}_{\text{temp}}^{\text{\ }\prime}}$$
$$\bigoplus \int_{J(\Gamma)\in \widehat{H}-\widehat{H}_{\text{temp}}^{\text{\ }\prime}} J(\Gamma)^{\oplus m(J(\Gamma),\pi|_H)} d\mu|_{\widehat{H}-\widehat{H}_{\text{temp}}^{\text{\ }\prime}}.$$
Every matrix coefficient of $\pi|_H$ can be decomposed into the sum of a matrix coefficient of the first representation plus a matrix coefficient of the second representation. As in Proposition 1.3 of \cite{How}, we deduce
$$\operatorname{WF}(\pi|_H)$$
is the union of the 
$$\operatorname{WF}\left(\int_{J(\Gamma)\in \widehat{H}_{\text{temp}}^{\text{\ }\prime}} J(\Gamma)^{\oplus m(J(\Gamma),\pi|_H)} d\mu|_{\widehat{H}_{\text{temp}}^{\text{\ }\prime}}\right)$$
and 
$$\operatorname{WF}\left(\int_{J(\Gamma)\in \widehat{H}-\widehat{H}_{\text{temp}}^{\text{\ }\prime}} J(\Gamma)^{\oplus m(J(\Gamma),\pi|_H)} d\mu|_{\widehat{H}-\widehat{H}_{\text{temp}}^{\text{\ }\prime}}\right).$$
The same goes for the singular spectrum of $\pi|_H$. Let 
$$q\colon i\mathfrak{g}^*\rightarrow i\mathfrak{h}^*$$
be the pullback of the inclusion map. We know from Proposition 1.5 of \cite{How} that 
$$\operatorname{WF}(\pi|_H)\supset \overline{q(\operatorname{WF}(\pi))}.$$
One checks that the proof of Proposition 1.5 of \cite{How} also yields the analogous statement for the singular spectrum. Now, by Theorem 1.1 of this paper, we deduce
$$\operatorname{WF}\left(\int_{J(\Gamma)\in \widehat{H}-\widehat{H}_{\text{temp}}^{\text{\ }\prime}} J(\Gamma)^{\oplus m(J(\Gamma),\pi|_H)} d\mu|_{\widehat{H}-\widehat{H}_{\text{temp}}^{\text{\ }\prime}}\right)\subset i\mathfrak{h}_s^*.$$
Again, the analogous statement holds for the singular spectrum as well. Thus, we deduce
$$\operatorname{WF}\left(\int_{J(\Gamma)\in \widehat{H}_{\text{temp}}^{\text{\ }\prime}} J(\Gamma)^{\oplus m(J(\Gamma),\pi|_H)} d\mu|_{\widehat{H}_{\text{temp}}^{\text{\ }\prime}}\right)\supset \overline{q(\operatorname{WF}(\pi|_H))}\cap i(\mathfrak{h}^*)'.$$
Again, this result still holds if we replace $\operatorname{WF}$ with $\operatorname{SS}$ everywhere. Finally, we apply Theorem 1.2 of \cite{HHO} to deduce
$$\operatorname{AC}\left(\bigcup_{\substack{\sigma\in \operatorname{supp}\pi|_H\\ \sigma\in \widehat{H}_{\text{temp}}^{\text{\ }\prime}}}\mathcal{O}_{\sigma}\right)\supset \overline{q(\operatorname{WF}(\pi|_H))}\cap i(\mathfrak{h}^*)'.$$
Since Theorem 1.2 of \cite{HHO} is also stated for the singular spectrum, we deduce the identical statement with $\operatorname{WF}$ replaced by $\operatorname{SS}$, which is the statement of Corollary \ref{restcor}.
\bigskip

Next, we consider the example noted in the introduction. Let $G=\operatorname{GL}(2n,\mathbb{R})$, let $H=\operatorname{SO}(n,n)$, and let $\pi$ be a Stein complementary series representation of $G$ (see \cite{St67} for the original definition of Stein complementary series; see \cite{Vo86} to understand how they fit into the unitary dual of $\operatorname{GL}(2n,\mathbb{R})$). Now, $\pi$ is (non-unitarily) parabolically induced from a parabolic subgroup with $\operatorname{GL}(n,\mathbb{R})\times \operatorname{GL}(n,\mathbb{R})$ as its Levi factor. Let $\mathfrak{p}$ denote the Lie algebra of this parabolic subgroup. Combining work of Barbasch-Vogan and Rossmann \cite{BV}, \cite{Ro95}, we check that $\operatorname{WF}(\pi)\supset i(\mathfrak{g}/\mathfrak{p})^*$.

Thus, to check $\overline{q(\operatorname{WF}(\pi))}=i\mathfrak{h}^*$, we need only check 
$$\overline{q(\operatorname{Ad}^*(G)\cdot i(\mathfrak{g}/\mathfrak{p})^*)}=i\mathfrak{h}^*.$$ To do this, it is enough to fix a standard set of representatives of the conjugacy classes of Cartan subalgebras of $\mathfrak{h}$, say $\{\mathfrak{b}_{\alpha}\}_{\alpha\in \mathcal{A}}$, and show that $$q(\operatorname{Ad}^*(G)\cdot i(\mathfrak{g}/\mathfrak{p})^*)\supset i\mathfrak{b}_{\alpha}^*$$ for every $\alpha$. One can use the set of representatives of conjugacy classes of Cartan subalgebras given in \cite{Su}. In addition, after identifying $\mathfrak{g}\cong i\mathfrak{g}^*$ equivariantly, one can take conjugates by permutation matrices of matrices consisting of scalar multiples of Jordan blocks which are also conjugate to matrices in $i(\mathfrak{g}/\mathfrak{p})^*$. One checks that one can obtain all of $i\mathfrak{b}_{\alpha}^*$ for every $\alpha\in \mathcal{A}$ by projecting these special matrices. We leave the details as an exercise for the reader.

Then by Corollary \ref{restcor}, we deduce 
$$\operatorname{AC}\left(i(\mathfrak{b}_{\alpha}^*)'-\operatorname{supp}\pi|_H\right)=i\mathfrak{b}_{\alpha}^*$$
for every $\alpha$. This justifies the claim that the tempered part of the decomposition of $\pi|_{\operatorname{SO}(n,n)}$ into irreducibles is ``asymptotically dense'' in $\widehat{H}_{\text{temp}}$. If $n$ is even and one takes the compact Cartan of $\operatorname{SO}(n,n)$, then one deduces the existence of infinitely many Harish-Chandra discrete series $\sigma$ of $\operatorname{SO}(n,n)$ for which
$$\operatorname{Hom}_{\operatorname{SO}(n,n)}(\sigma,\pi|_{\operatorname{SO}(n,n)})\neq \{0\}.$$

\section{Acknowledgements}
The author would like to thank Hongyu He and Gestur \'{O}lafsson, his postdoctoral advisers at Louisiana State University. Hongyu and Gestur taught the author many things during his time in Baton Rouge. In particular, the first paper in this series was coauthored with the two of them \cite{HHO}. 

Second, the author would like to thank Toshiyuki Kobayashi for several discussions on the relationships between his excellent papers \cite{BK15}, \cite{Ko98} and this article. Toshi's insights are a continual inspiration to the author.

Third, the author would like to thank David Vogan for answering a question about the unitary dual of a real, reductive algebraic group.

\bibliographystyle{amsalpha}
\bibliography{WFSetsII}

\end{document}